\documentclass[11pt]{article}

\usepackage{libertine}
\usepackage[T1]{fontenc}
\usepackage{amsmath,amsthm,mathtools}
\usepackage{graphicx}
\usepackage{verbatim}
\usepackage{xcolor}
\usepackage[margin = 1in]{geometry}

\usepackage{amsfonts} 
\usepackage{bm}
\usepackage{subfigure}
\usepackage{float}

\newcommand{\Eb}{\mathbb{E}}
\newcommand{\Nb}{\mathbb{N}}

\newcommand{\Rb}{\mathbb{R}}

\newcommand{\Fc}{\mathcal{F}}
\newcommand{\Ff}{\mathfrak{F}}
\newcommand{\Hc}{\mathcal{H}}
\newcommand{\Jc}{\mathcal{J}}

\newcommand{\Lc}{\mathcal{L}}

\newcommand{\Pc}{\mathcal{P}}
\newcommand{\Qc}{\mathcal{Q}}
\newcommand{\Uc}{\mathcal{U}}
\newcommand{\Xc}{\mathcal{X}}

\DeclareMathOperator{\cov}{cov}
\DeclareMathOperator{\Var}{Var}

\DeclareMathOperator{\intt}{int}

\newcommand{\dto}{\xrightarrow[n\to\infty]{d}}

\newtheorem{proposition}{Proposition}[section]
\newtheorem{theorem}[proposition]{Theorem}
\newtheorem{corollary}[proposition]{Corollary}

\newtheorem{definition}[proposition]{Definition}

\begin{document}

\title{Central limit theorems for vector-valued composite functionals with smoothing and applications
\thanks{The support of the Office of Naval Research under grant 
N00014-21-1-216 is gratefully acknowledged.}
}

\author{Huihui Chen, Darinka Dentcheva, Yang Lin, and Gregory J. Stock\\
Stevens Institute of Technology, Hoboken, NJ, USA}

\maketitle

\textbf{Abstract:}

This paper focuses on vector-valued composite functionals, which may be nonlinear in probability. Our primary goal is to establish central limit theorems for these functionals when mixed estimators are employed.
Our study is relevant to the evaluation and comparison of risk in decision-making contexts and extends to functionals that arise in machine learning methods. A generalized family of composite risk functionals is presented, which encompasses most of the known coherent risk measures including systemic measures of risk. The paper makes two main contributions. First, we analyze vector-valued functionals, providing a framework for evaluating high-dimensional risks. This framework facilitates the comparison of multiple risk measures, as well as the estimation and asymptotic analysis of systemic risk and its optimal value in decision-making problems. Second, we derive novel central limit theorems for optimized composite functionals when mixed types of estimators: empirical and smoothed estimators are used. We provide verifiable sufficient conditions for the central limit formulae and show their applicability to several popular measures of risk.

\textbf{Keywords}
coherent measure of risk, stochastic programming, systemic risk

\section{Introduction}

In the area of machine learning, business, engineering, and others, optimization under uncertainty and risk are indispensable. A plenitude of literature addresses the properties of and efficient numerical approach to data-driven stochastic optimization problems. Recently, the methods of risk-averse optimization and learning have become a subject of increased interest and our paper aims to contribute to that area. 

Our main focus is placed on the following general functions.
\begin{equation}
\label{f:rho}
\varrho[X] =  \mathbb{E}\left[f_1 \left(\mathbb{E}[f_2(\mathbb{E}[\ldots f_k(\mathbb{E}[f_{k+1}(X)],X \right) \right] \ldots,X)],X)],
\end{equation}
where $X$ is a random vector defined on the probability space $(\Omega,\mathcal{F},\mathbb{P})$ with realizations in $\Xc\subseteq{\mathbb{R}}^m$.  The probability measure induced by the random vector $X$ is denoted by $P$.
The vector functions $f_j:{\mathbb{R}}^{m_j} \times {\mathbb{R}}^m \rightarrow {\mathbb{R}}^{m_{j-1}}$, $j=1,\cdots,k$ and $f_{k+1}:{\mathbb{R}}^m \rightarrow {\mathbb{R}}^{m_k}$ are assumed $P$-integrable with respect to their last argument and for $j=1,\dots, k$ they are continuous with respect to the first argument. 
The standard notation $\Lc_p(\varOmega,\mathcal{F},\mathbb{P};\Rb^m)$ stands for the set of $m$-dimensional random vectors defined on the probability space $(\varOmega,\mathcal{F},\mathbb{P})$  that are indistinguishable on sets of $\mathbb{P}$-measure zero and have finite $p$ moments.
We assume that $X\in\Lc_p(\varOmega,{\mathcal S},\mathbb{P};\Rb^m),$ with some $p\in[1,\infty]$.

Many coherent measures of risk may be cast in this form. In \cite{dentcheva2017statistical}, we have shown that the mean-semi deviations measures of order $p\geq 1$,  the Average Value at Risk at level $\alpha\in(0,1]$, as well as the higher-order measures of risk  can be represented as (optimized) composite functionals of form \eqref{f:rho} with $u$ being a decision vector when optimization of risk control is involved. For more information on these measures, we refer to \cite{Ogryczak_2001,Ogryczak_2002,ROCKAFELLAR20021443,krokhmal2007higher}. A comprehensive treatment of optimization models with risk measures is provided in \cite{mainbook,Dentcheva2024}.
Furthermore, problems in other areas, such as machine learning, deal with composite optimization as well \cite{machine1,nest3}. 

Suppose a sample $X_1, X_2,\dots, X_n$ of $n$ independent realizations of the
random vector $X$ is available. We are interested in the case when we need to use the entire sample for estimating the expected values at all levels. This need arises when we do not have a large sample for the given dimension $m$ of $X$ and sampling is expensive or difficult.
The empirical estimator of the composite risk functional is the following
\begin{equation}
\label{f:rhoN}
\varrho^{(n)}[X]= \sum_{i_0=1}^n\frac{1}{n}\Big[f_1\Big(\sum_{i_1=1}^n\frac{1}{n}\big[f_2\big(\sum_{i_2=1}^n\frac{1}{n}[\cdots  f_k(\sum_{i_k=1}^n\frac{1}{n}f_{k+1}(X_{i_k}),X_{i_{k-1}})]
\cdots,X_{i_1}\big)\big],X_{i_0}\Big)\Big]\qquad{~}
  \end{equation}
When $m_0=1$, we shall consider optimizing a risk measure. In that case, we shall assume that all functions depend additionally on a decision variable $u$ and we solve the following problem. 
\begin{equation}\label{eq_2}
\vartheta[X]= \min_{u \in U} \varrho[u,X],
\end{equation}
assuming that $U$ is closed convex set in $\Rb^d$. Let $\mathcal{S}$ be the set of optimal solutions of problem \eqref{eq_2}. Our assumptions will guarantee  that $\mathcal{S} \not=\emptyset.$
When the risk measure is estimated based on a sample, then the optimal value 
becomes an estimator itself.

In our study, we intend to pursue the analysis of \emph{vector-valued} composite functionals with the use of various estimators beyond empirical ones and to provide analysis that are novel also for the univariate case. The need to introduce vector-valued functionals arises in several contexts. 
First, in the context of investment decisions, 
we may want to compare the riskiness of two different investment portfolios based on observed data, which may or may not come from the same basket of securities. Furthermore, we may need to compare the risk of the two portfolios with respect to several measures of risk, which reflect the preferences of multiple investors.
A very important case, when we need to use vector-valued risk functionals is when evaluating and optimizing a complex stochastic system. In these situations, we deal with systemic risk, which is very essential in both financial as well as engineering, logistics, medical, and many other problems. In those applications, the decision maker deals with complex distributed systems, where each component (unit, or agent) has its own risk of operation.  
It is well-known that the risk is not additive and various risk aggregation methods are suggested in the literature to reflect the risk of the entire system; the latter is termed \emph{systemic risk}. Most practical approaches to systemic risk evaluation are based on (weighted) linear or nonlinear aggregations of risks of the system's components.  The decision problems optimizing such complex distributed systems incorporate measures of risk for each agent (unit), as well as a measure of systemic risk associated with a common task, integrity of the system, etc. In \cite{almen2023risk}, these aggregation methods are discussed and it is shown that they are in-line with an axiomatic foundation of high-dimensional risks. The statistical evaluation of the systemic risk as well as the need to work with more than one measure of risk leads to the vector-valued setting, which we discuss in this work. 

Given the observations of our earlier work \cite{dentcheva2021bias,dentcheva2022stability}, we
aim at analyzing the properties of the smoothed estimators, as well as the mixed smoothed and empirical ones, at a deeper level and address some of the nested expected value functionals for the vector-valued case. In this paper, we establish central limit formulae which are novel for smoothed and mixed estimators for both scalar-valued and vector-valued composite functionals. We underline that our setting has the flexibility to employ also other estimators across different layers using the entire sample; our findings are not tied to a single estimator. Smoothing is a widely used technique both in the context of statistical estimation as well as in the context of randomized optimization methods (e.g., \cite{duchi2012randomized}, \cite{Ermoliev1995}). We have also shown that smoothing has the potential for bias reduction in stochastic optimization (\cite{dentcheva2021bias,dentcheva2022stability}).  The results in this paper may further the convergence analysis of such techniques.  


Finally, we apply our technique and results to compare differences in risk measures and estimate systemic risk in various forms.

Our paper is organized as follows. Section \ref{s:framework} introduces the foundational framework and presents an overview of the existing univariate central limit theorems pertaining to empirical estimators for composite risk functionals. 
We also define the notion of strong approximate identity, which is germane to the analysis of the smoothed estimators. We pay particular attention to the kernel estimators. Section \ref{s:CLT} contains a central limit formula for general smoothed estimators of scalar-valued composite risk functionals, which is based on verifiable assumptions without assuming the boundedness of the functions involved, or non-negativity of the smoothing measures. Additionally, we show that the kernel estimators represent a strong approximate identity under very mild conditions. We further extend these results to vector-valued composite risk functionals and examine two compelling applications. Section \ref{s:simulation} contains the results of our simulation study, which compares the kernel estimator against the empirical estimator to gauge the precision of our approximation. Finally, Section \ref{s:conclusions} summarizes our conclusions.

\section{Framework and Preliminaries}
\label{s:framework}
Let $\Pc(\Xc)$ be the set of all probability measures on the set $\Xc$.
 The following functions and sets will play a role in our discussion. For a measure $Q\in\Pc(\Xc)$, we define
\begin{equation}
\label{f:barfj-u-mu}
\begin{aligned}
     \bar{f}_j^Q(\eta_j) & =\int_{\mathcal{X}}f_j(\eta_j,x)\,{Q}(dx), \quad j=1,\cdots,k, \\
     \bar{\eta}^Q_{k+1} &  =\int_{\mathcal{X}}f_{k+1}(x)\,{Q}(dx),\\
     \bar{\eta}_{j}^Q & =\bar{f}_j^Q(\bar{\eta}_{j+1}^Q), \quad j=1,\cdots,k.
\end{aligned}
\end{equation}
If $Q=P$, i.e., the expectation is exact, not approximated by using measure $Q$, then the superscript will be omitted. 
Notably, different estimators could be employed across different layers. These encompass not only empirical or various smoothed estimators but also their amalgamation to yield more refined estimation results. The composite risk function is expressed through a combination of multiple estimators, as follows:

\begin{equation}\label{multi_est}
\varrho^{(\Qc)}[X] = f_{1}^{Q^1} \left(f_{2}^{Q^2}(\ldots f_{k}^{Q^k}(f_{k+1}^{Q^{k+1}}(X),X )\ldots,X),X)\right)
\end{equation}
Here, $\Qc=(Q^1, Q^2, \dots, Q^{k+1})$ denotes a set of distinct estimators of the respective expected values across various layers. 
We shall denote the convergence in distribution by the symbol $\xrightarrow[]{d}.$

We fix compact sets $I_1\subset\Rb^{m_1},\cdots,I_k\subset\Rb^{m_k}$ such that $\bar{f}_{j+1}(I_{j+1}) \subset \intt(I_j)$, $j=1,\cdots,k-1$, and $\bar{f}_{k+1}\subset \intt(I_k)$, where $\intt(I_j)$ stands for the interior of $I_j.$ Without loss of generality, we assume that $I_j$, $j=1,\dots k$ are convex sets. We define the space:
\begin{equation*}
    \mathcal{H}={\mathcal{C}}_1(I_1) \times {\mathcal{C}}_{m_1}( I_2) \times \cdots \times {\mathcal{C}}_{m_{k-1}}(I_k) \times \Rb^{m_k}.
\end{equation*}
where  ${\mathcal{C}}_{m_{j-1}}$ is the space of ${\mathbb{R}}^{m_{j-1}}$-valued continuous function on $I_j$, equipped with the supremum norm. The space $\Hc$ is equipped with the product norm.
We set $I = I_1\times I_2\times\cdots I_k$ and $M=m_0+m_1+\dots+m_k$. Let the vector-valued function
$
f: I \times \Xc \to \Rb^M
$
have block coordinates
$f_j(\eta_j,x)$, $j=1,\dots,k$, and $f_{k+1}(x)$. Similarly, we define $\bar{f}^\Qc: I  \to \Rb^{M}$ with block
coordinates $\bar{f}_j^{Q^j}(\eta_j)$, $j=1,\dots,k$, and $\bar{\eta}^Q_{k+1}$.

Recall that a function $g$ is Hadamard-directionally differentiable at a point $x$ in a direction $d$ if for all sequences $\{ d^k\}_{k=1}^\infty$ and $\{ t_k\}_{k=1}^\infty$ such that
$\lim_{k\to\infty} d^k = d$ and $t_k>0$, $\lim_{k\to\infty} t_k =  0$, the limit 
\[
\lim_{k\to\infty} \frac{1}{t_k} \big( g(x+t_k d^k) - g(x)\big) = g'(x;d)
\]
exists and is well-defined. Compositions of Hadamard-directionally differentiable functions are Hadamard-directionally differentiable. In our setting
for every direction $d = (d_1,\dots,d_{k},d_{k+1}) \in\Hc$, we define recursively the sequence of vectors:
\begin{equation}
\label{recursive0}
\begin{gathered}
\xi_{k+1}(d) = d_{k+1},\\
\xi_{j}(d) = \int_{\Xc} f'_{j}\big(\bar{\eta}_{j+1},x;\xi_{j+1}(d)\big)\,P(dx) + d_{j}\big(\bar{\eta}_{j+1}\big),\quad j=k,k-1,\dots,1.
\end{gathered}
\end{equation}
A Central Limit Theorem for the plug-in estimator of a univariate composite risk functional is proved in \cite{dentcheva2017statistical}. 
For the risk-averse optimization problem of form \eqref{eq_2}, the counterpart central limit formula is established in \cite[Theorem 3]{dentcheva2017statistical} assuming that only empirical estimators are used.

We pay special attention to estimators that are obtained by convolution of the empirical measure $P_n$ with a measure $\mu_n$. One of our goals is to obtain a central limit formula for this type of estimators when they are used in compositions. 

A comprehensive examination of the asymptotic behavior of smoothed empirical processes is available in \cite{Aad_1994}, \cite{yukich_1992}, and \cite{YUKICH1989163}. Thorough explorations of the one-dimensional functional central limit theorem for smoothed empirical processes, assuming a uniformly bounded class of functions or invariance under translation (which in turn implies uniform boundedness) are extensively covered in \cite{Einmahl2005}, \cite{Gine2008}, and \cite{Radulovic2003}. The requirement of uniform boundedness is relaxed in \cite{Peter1998} and \cite{Rost1} by assumptions about the entropy and the tail behavior for smoothed estimators. These assumptions are quite involved in may not be easy to verify. In our context, we need to further propagate the relevant properties through the compositions which form the risk functionals.  To resolve these issues, we offer sufficient conditions, which are verifiable and allow for the extension of the analysis to vector-valued composite stochastic optimization problems. Our starting point is our earlier work \cite{dentcheva2022stability}, where we have considered several smoothed estimators for the composite risk functionals and have shown their consistency under relatively mild assumptions. 

We choose a sequence of measure $\{\mu_n\}_{n=1}^\infty$, which are independent of the empirical measure $P_n$, and throughout the entire paper, we assume that all measures $\{\mu_n\}$ are normalized, i.e. $\mu_n(\Rb^m)=1$. The following notion will be used.
\begin{definition}
The sequence of measures $\{\mu_n\}$ is called  a \emph{proper approximate convolutional identity of order $p$}, $p\geq 1,$ if it converges weakly to the point mass $\delta(0)$ when $n\to\infty$ and the integrals $\int_{\Rb^m} \|z\|^p \,d\mu_n(z)$ are finite for all $n\in\Nb$. 
\end{definition}
The kernel estimators of the following form constitute a special case:
\begin{equation*}
    \frac{1}{nh_n^m}\sum_{i=1}^n \int_{\Rb^m} g(x) K\Big(\frac{x-X_i}{h_n}\Big)\, dx
\end{equation*}
where $K$ is a $m$-dimensional density function with respect to the Lebesgue measure and $h_n>0$ is a smoothing parameter such that $\lim_{n\to\infty} h_n=0$. We have
 $d\mu_n(x) = \frac{1}{h_n^m} K\Big(\frac{x}{h_n}\Big)\,dx$. 
The estimators $\mu_n$ may take a more general form than the kernel estimator just defined for illustration (cf. \cite{tsybakov2008introduction,gine2016mathematical}).
When using kernels, we shall assume the following properties.
\begin{itemize}
\item[(k1)] The kernel $K$ of order $s>1$ is a density function with respect to the Lebesgue measure satisfying the symmetry condition 
$\int\limits_{\Rb^m} y_l^j K(y)dy=0$ for $l=1,\cdots,m$, $j=1,\dots, \lfloor s\rfloor$ with $\lfloor s\rfloor$ being the largest integer smaller than $s.$
\item[(k2)]  The $p$-th order moment of the kernel: 
$m_p(K) = \int\limits_{\Rb^m}\|y\|^p K(y)dy$,  is finite.
\end{itemize}

For illustration, in the case of $k=1$, we deal with a functional with two layers. We could use the empirical estimator as $Q^1_n=P_n$ in the outer layer  and a kernel estimator as $Q^2_n$ in the inner layer. In that case, the two-layer estimation based on a finite sample of size $n$ and $\Qc=(Q^1_n,Q^2_n)$ is represented as follows:
\begin{align*}
      \varrho^{(\Qc)}[X]& =f_{1}^{Q^1_n}\left(f_{2}^{Q^2_2}(X),X\right)\\
      &=\sum_{i_0=1}^n\frac{1}{n}\Big[f_1\Big(\frac{1}{nh_n^{m_1}}\sum_{i_1=1}^n \int_{\Rb^{m_1}} f_2(x) K\Big(\frac{x-X_{i_1}}{h_n}\Big)\, dx, X_{i_0})\Big]
\end{align*}

In order to avoid cluttering the notation, we shall omit the area of the integration when it does not lead to ambiguity.  We shall denote the index set $\Jc =\{1,2,\dots,k+1\}$ and the set $J\subseteq\Jc$ shall contain all indices of the composition level, where smoothing is applied. We use the notation $\varrho_\mu^{(n,J)}$ for the estimator, in which $Q_n^j = P_n *\mu_n$ for $j\in J$ and $Q_n^j = P_n$ for $j\in \Jc\setminus J.$

\section{Central limit theorems for mixed smoothed and empirical estimators}
\label{s:CLT}

\subsection{Scalar-valued composite risk functionals}

In this section, we establish a central limit theorem for univariate composite risk functionals and their optimized version.
We define the following sets of continuous functions associated with the risk functional \eqref{f:rho}:
\begin{align*}
\mathfrak{F}_j\,&= \Big\{  f_{j,i}(\eta_j,\cdot): \Xc\rightarrow \Rb, \;
 i=1,\dots m_{j-1},\;\; \eta_j\in I_j\; \Big\}  \quad  j=1,\dots, k; \\
\mathfrak{F}_{k+1}\;& = \Big\{ f_{k+1,i}(\cdot): \Xc\rightarrow \Rb, \;\; i=1,\dots,m_{k+1},\; \Big\} 
\quad \mathfrak{F} =\cup_{1\leq j\leq k+1} \mathfrak{F}_j. 
\end{align*}
The associated envelope functions is given by
\[
F_j(x)= \sup_{\eta_j \in I_j,\; 1\leq i\leq m_{j-1}} | f_{j,i} ( \eta_j, x) |,\quad F_{k+1}(x)= \sup_{1\leq i\leq m_{j-1}} | f_{k+1,i} (x) |
\]
Due to the compactness of $I_j$, the functions $F_j(\cdot)$, $j\in\Jc$ are well-defined and they are also measurable (\cite[Theorem 7.42]{mainbook}). 
We do not assume $\Ff$ is a translation invariant class, which is a common assumption in the 
literature (\cite{Aad_1994},\cite{tucl} and \cite{yukich_1992}) because this assumption automatically implies that $\Ff$ is uniformly bounded, which would limit substantially the intended application of our results. 

We recall the notions of covering and bracketing numbers. 
The \emph{covering number} $\mathcal{N} (\varepsilon, \mathfrak{F}, \| \cdot \|)$ is the minimal number of balls $\{ g,f:\Rb^m : \| g - f \| < 2\varepsilon \}$ of radius $\varepsilon$ needed to cover the set $\mathfrak{F}$. 

  Given two functions $l:\Rb^m\to \Rb$ and $u:\Rb^m\to \Rb$, the bracket $[l, u]$ is
  the set of all functions  $g:\Rb^m\to \Rb$ with $l \leq g \leq u$. An $\varepsilon$-bracket is a bracket $[l, u]$ with $\| u - l \| < \varepsilon$. The bracketing number
  $\mathcal{N}_{[\,]} (\varepsilon, \mathfrak{F}, \| \cdot \|)$ is the minimum number of $\varepsilon$-brackets needed to cover $\mathfrak{F}$.

Suppose that the envelope functions satisfy $PF_j<\infty$ for $j=1,\dots k+1.$
Then the classes $\mathfrak{F}_j$, $j = 1, \ldots, k+1$ have finite bracketing
numbers $\mathcal{N}_{[\,]} (\varepsilon,\mathfrak{F}_j, \Lc_1 (P))$ for all $\varepsilon >0$.
This implies that $\mathfrak{F}$ is a $P$-Glinvenko-Cantelli class, i.e. 
\begin{equation}
\label{Glinvenko-Cantelli}
\|P_nf-Pf\|=\sup_{\eta \in  I}|P_nf-Pf| \xrightarrow[n\to\infty]{a.s.} 0.
\end{equation}
 We propose the following central limit theorem for smoothed composite vector-valued functions and provide more details for its covariance. 

\begin{theorem}\label{smoothclt}
  Suppose an index set $J\subseteq\{1, \ldots, k+1\}$ is fixed, $m_0 = 1$, and the following conditions are satisfied:
  \begin{enumerate}
     \item[(a1)] The sequence of normalized measures $\{ \mu_n \}_{n=1}^\infty$ converges weakly to the point mass at zero.
    \item[(a2)] For all $j\in\Jc$,  $\int_{\mathcal{X}} \|F_j(x)\|^2 P(d x)$ is finite and for all $j\in J$ and $x \in \mathcal{X}$, $\int_{\mathbb{R}^m} \|F_j(x+y)\|^2 \mu_n (dy) < \infty$ is finite. 
    \item[(a3)] Functions $\tilde{F}_j:\Xc\to \Rb$, $j\in\Jc$ with a finite $\Lc_2 (P)$ norm exists such that for all $\eta,\eta'\in I_j$ and all $x\in\Xc$
\begin{gather*}
\|f_j(\eta,x)-f_j(\eta',x)\| \leq \|\eta -\eta'\|\tilde{F}_j(x),\quad j=1,\dots, k. 
\end{gather*}
    \item[(a4)]
    The following two conditions are satisfied for all $j\in J$:
    \begin{gather}\label{eq.15}
      \sup\limits_{f \in \mathfrak{F}_j}\; \mathbb{E} \Big[ \Big(
      \int_{\mathbb{R}^m} \big(f (\eta, X + y) - f(\eta, X)\big) \;
      \mu_n (dy) \Big)^2 \Big]  \xrightarrow[n \rightarrow \infty]{} 0;
      \\
    \label{eq.16}
      \underset{f \in \mathfrak{F}_j}{\sup} \; \sqrt{n} 
      \Big| \mathbb{E} \Big[ \int_{\mathbb{R}^m} \big( f (\eta, X + y) -f (\eta, X) \big) \;\mu_n (dy)  \Big] \Big| \xrightarrow[n \rightarrow \infty]{} 0.
    \end{gather}
    \item[(a5)] The functions $f_j (\cdot, x)$, $j=1,\dots, k$ are Hadamard directional differentiable for all $x \in \mathcal{X}$ and for $j\in J$ their directional derivatives are $\mu_n$-integrable for all directions $d$.
  \end{enumerate}
 Then 
$\displaystyle{\sqrt[]{n}\big( \varrho_{\mu}^{(n,J)}[X]-\varrho[X]\big) \xrightarrow{d} \xi_1(G)}$,
where $G(\cdot)=(G_1(\cdot),\ldots,G_k(\cdot),G_{k+1})$ is zero-mean Brownian process on $I$. Here $G_j(\cdot)$ is a Brownian process of dimension $m_{j-1}$ on $I_j$, $j=1,\ldots,k$, and $G_{k+1}$ is an $m_k$-dimensional normal vector. The covariance function of $G$ has the following form:
\begin{equation} \label{eq__12}
\begin{split}
& \cov[G_i(\eta_i),G_j(\eta_j)]=\int_{\mathcal{X}}[f_i(\eta_i,x)-\bar{f}_i(\eta_i)][f_j(\eta_j,x)-\bar{f}_j(\eta_j)]^{\top} {P}(dx) \\
& \qquad\qquad\qquad\qquad\qquad\qquad \eta_i \in I_i, i=1,\ldots,k, \\
&	\cov[G_i(\eta_i),G_{k+1}]=\int_{\mathcal{X}}[f_i(\eta_i,x)-\bar{f}_i(\eta_i)][f_{k+1}(x)-\bar{\eta}_{k+1}]^{\top} {P}(dx) \\
&	\cov[G_{k+1},G_{k+1}]=\int_{\mathcal{X}}[f_{k+1}(x)-\bar{\eta}_{k+1}][f_{k+1}(x)-\bar{\eta}_{k+1}]^{\top} {P}(dx)
\end{split}   
\end{equation}
\end{theorem}
\begin{proof}
We shall show that the classes $\Ff_j$ with $j\in J$ have uniformly integrable entropy. 
To this end, it is sufficient to show that for some constant $V$ and all $\varepsilon\in(0,1)$, the following bound holds:
    \[ 
    \underset{Q}{\sup}\; \mathcal{N} (2\varepsilon\|\tilde{F}\|_{\Lc_2(Q)},
       \mathfrak{F}_j, \Lc_2 (Q)) \leq K \left( \frac{1}{\epsilon} \right)^V.
       \]
    Here the supremum is taken over all finitely discrete probability measures
    $Q$ on $\mathcal{X}$ such that  
    \[
    \| \tilde{F} \|_{2, Q} = \int \tilde{F}^2(x)\; Q(dx) > 0.
    \]
Using assumption (a3), we observe that for any norm $\|\cdot\|$ such that $\|\tilde{F}\|< \infty$,  the following bound hold:
\begin{equation}
\label{bracketbound}
\mathcal{N}_{[\,]} \big(2\varepsilon\|\tilde{F}\|,\mathfrak{F}_j, \|\cdot\|\big) \leq N(\varepsilon, \Uc\times I,\|\cdot\|_{\Rb} ),
\end{equation}
where $N(\varepsilon, I,\|\cdot\|_{\Rb} )$ denotes the minimal number of balls with radius $\varepsilon$ that are necessary to cover $I$. 
Indeed, let $\eta_\ell$, $\ell = 1,\dots,\ell_\varepsilon$ with $\ell_\varepsilon = N(\varepsilon, I,\|\cdot\|_{\Rb} )$ form an $\varepsilon$-net on $I$, that is 
the closed balls with centers $\eta_\ell$ and radius $\varepsilon$ cover $I$. 
Then the brackets 
\[
\big[ f_{ij}(\eta_\ell,\cdot) -\varepsilon \tilde{F}_j(\cdot), f_{ij}(\eta_\ell,\cdot) +\varepsilon \tilde{F}(\cdot) \big]
\] 
cover $\mathfrak{F}_j$, and they are of size at most $2\varepsilon\|\tilde{F}\|$. 
According to Lemma 2.7.8 in \cite{van1996}, there is an universal bound for the logarithm of the covering number of all compact convex subsets of $I_j$ given by $K\left(\frac{1}{\varepsilon}\right)^{\frac{m_j-1}{2}}$ with a constant $K>0$, which depend only on the volume and dimension of $I_j$. 
Furthermore, since the covering numbers are smaller than the bracketing numbers, we have 
    \[
    \log \mathcal{N} \left( 2\varepsilon \| \tilde{F}_j\|_{2, Q}, \mathfrak{F}_j, \Lc_2 (Q) \right) 
       \leq \log \mathcal{N}_{[\,]}(2\varepsilon \| \tilde{F}_j \|_{2, Q}, \mathfrak{F}_j, \Lc_2(Q))
        \leq K \left( \frac{1}{\epsilon} \right)^{\frac{m_j-1}{2}}\hspace{-6ex}
    \]  
This shows that the functions in $\Ff$ have uniformly integrable entropy. This together with conditions (a1), (a2) and (a4) implies that the assumptions of \cite[Theorem 2.2]{Rost} are satisfied. Note that we do not require $F_j(\cdot)$ to be included in $\mathfrak{F}_j$ as in \cite[Theorem 2.2]{Rost} because (a2) insures the necessary integrability. Hence, the class of functions $\mathfrak{F}_j$ is $P \ast \mu_n$ - Donsker.
The classes $\Ff_j$ with $j\in \Jc\setminus J$ are $P$-Donsker due to assumption (a3). 
This entails that
  \begin{gather*}
  \sqrt{n} [(P_n \ast \mu_n) f_j - P f_j] \overset{d}{\rightarrow} G^P_j \quad j\in J,\\
 \sqrt{n}\quad  [P_nf_j - P f_j] \overset{d}{\rightarrow} G^P_j,\quad j\in\Jc\setminus J.
     \end{gather*}
 where ${G}^{P}_j$, $j\in\Jc$ is a standard Brownian process with zero-mean and covariance function 
\begin{equation} \label{eq__13}
\cov[G({\eta}^{'},{\eta}^{''})]=\int_{\mathcal{X}}[f_j(\eta^{'},x)-\bar{f}_j(\eta^{'})][f_j(\eta^{''},x)-\bar{f}_j(\eta^{''})]^\top {P}(dx)
\end{equation}
We define a subset $H$ of the space $\Hc$ containing all elements $(h_1,\dots,h_k, h_{k+1})$ for which\\
$h_{j+1}(h_{j+2}(\cdots h_k(h_{k+1}) \cdots ))\in~I_j$, $j=1,\dots,k$. 
Further, we define the operator $\Psi: H \to \Rb$ as follows
\[
\varPsi(h) = h_1\Big(h_2\big(\;\cdots h_k(h_{k+1})\;\cdots\big)\Big).
\]
Consider the perturbation function $h_\mu^{(n,J)}(\eta)$ as follows: its $j$-th component is given by
\begin{equation}
\label{e:hmu}
[h_\mu^{(n, J)}]_j(\eta)=
\begin{cases}
\frac{1}{n}\sum_{i=1}^{n}\int f_j(\eta,X_i+ z) d\mu_n(z) &\text{ if } j\in J, j=1,\dots, k,\\
\frac{1}{n}\sum_{i=1}^{n} f_j(\eta,X_{i}) &\text{ if } j\not\in J, j=1,\dots, k,\\
\frac{1}{n}\sum_{i=1}^{n}\int f_{k+1}(X_{i}+ z) d\mu_n(z) &\text{ if } k+1\in J,\\
\frac{1}{n}\sum_{i=1}^{n} f_{k+1}(X_{i}) &\text{ if } k+1\not\in J.\\
\end{cases}
\end{equation}
We define the functional $\Psi:\Hc \to\Rb$ by setting
\[
\Psi(h) = h_1\Big(h_2\big(\cdots h_k(h_{k+1})\cdots\big)\Big)
\]   
The Hadamard-directional differentiability of $\Psi(\cdot)$ at $\bar{f}$ is needed in order to apply the delta method. 

Note that $\bar{f}$ is an element of $\mathcal{H}$  and it is also an interior point of $H$ due to assumption (a3). 
The perturbations $h^{(n)}_\mu(\cdot)$ are in the space $\mathcal{H}$ as well. We need them to be Hadamard-directionally differentiable. To this end, we only need consider the components with index $j\in J$ and observe that it is sufficient to argue that $\int f_j(\cdot, X_i+z) \mu_n(dz)$ is  Hadamard-directionally differentiable. Using the definition, we have
\begin{multline*}
\lim_{\substack{\tau \downarrow 0 \\ d \rightarrow \xi_{j+1}}} \frac{1}{\tau}\int \big(f_j(\bar{\eta}+\tau d, X_i+z)-f_j(\bar{\eta}, X_i+z) \big) \;\mu_n(dz) =\\
\int \lim_{\substack{\tau \downarrow 0 \\ d \rightarrow \xi_{j+1}}} \frac{1}{\tau} \big(f_j(\bar{\eta}+\tau d, X_i+z)-f_j(\bar{\eta}, X_i+z) \big) \;\mu_n(dz) 
        = \int f_{j}^{'}(\bar{\eta},X_i+z;\xi_{j+1})\; \mu_n(dz).
\end{multline*}
We could take the limit under the integral by virtue of the Lebesgue dominated convergence theorem due to assumption (a2). The integral on the right hand side is finite by virtue of assumption (a5).

 An explicit formula of the Hadamard-differentiable derivative $\Psi_k^{'}(h;d)$ can be derived as follows.   Consider $ \Psi_k(h)=h_k(h_{k+1}(u))$ and $h \in \intt(\mathcal{H})$. For $l \rightarrow \infty$, $d^l=(d^l_1,\ldots,d_k^l,d_{k+1}^l) \rightarrow d \in \mathcal{H}$. Then if $t_l \rightarrow 0$, $l \rightarrow \infty$, we have
    	\begin{equation}
         \begin{split}
           \Psi_k^{'}(h;d)
          &= \lim_{l \rightarrow \infty} \frac{1}{t_l}[\Psi_k((h_k+t_l d_k^l),(h_{k+1}+t_l d_{k+1}^l))-\Psi_{k}((h_k),(h_{k+1}))] \\
         &=  \lim_{l \rightarrow \infty } \frac{1}{t_l}[[h_k+t_l d_k^l](h_{k+1}+t_l d_{k+1}^l)-h_k(h_{k+1})]\\
         &=\lim_{l \rightarrow \infty} \frac{1}{t_l}[h_k(h_{k+1}+t_l d_{k+1}^l)-h_k(h_{k+1})+t_l d_k^l(h_{k+1}+t_l d_{k+1}^l)]\\
         &=\lim_{l \rightarrow \infty} \{ \frac{1}{t_l}[h_k(h_{k+1}+t_l d_{k+1}^l)-h_k(h_{k+1})]+ d_k^l (h_{k+1}+t_l d_{k+1}^l)  \} \\
         &= h_k^{'}((h_{k+1}),d_{k+1})+d_k(h_{k+1})
        \end{split}
    	\end{equation}
The chain rule entails that  $\Psi_{k-1}^{'}(h;d)=h_{k-1}^{'}(\Psi_k(h),\Psi_k^{'}(h;d))+d_{k-1}(\Psi_k(h))$.
This fact allows us to use the delta theorem presented in \cite{delta} to transfer the convergence of $h^{(n,J)}_{\mu}$ to the convergence of $\Psi(h^{(n,J)}_{\mu})=\varrho_{\mu}^{(n,J)}$.
The Delta Theorem \cite{delta}, the Donsker property, and the Hadamard directional differentiability of $\varPsi(\cdot)$ at $\bar{f}$
imply that
 \begin{equation}
     \sqrt{n}[\varrho^{(n,J)}_{\mu}-\varrho]=\sqrt{n}[\Psi(h^{(n,J)}_{\mu})-\Psi(\bar{f})] \xrightarrow{d} \Psi^{'}(\bar{f},G)
 \end{equation}
The covariance structure \eqref{eq__12} of $G$ follows directly from (\ref{eq__13}).
\end{proof}
Relations \ref{eq.15} and \ref{eq.16} in Theorem~\ref{smoothclt} play a crucial role in obtaining a central limit theorem. We provide a sufficient condition that may be easier to verify in the context of composite functionals.
To this end, we introduce the following notion. 
\begin{definition}
\label{d:strong_aprox_identity}
We call a proper approximate convolution identity $\{\mu_n\}$ a \emph{strong approximate identity of order} $p\geq 1$, if 
\begin{equation}\label{cond(c)}
 \sqrt{n} \lim\limits_{n\to\infty}  \int_{\Rb^m} \max\big(\|z\|,\|z\|^{p}\big)\; d \mu_{n}(z) =0.  
 \end{equation}
\end{definition}
Observe that when \eqref{cond(c)} is satisfied, then \cite[Definition 6.8 (i)]{villani2009optimal} implies that $\mu_n$ converges to $\delta(0)$ in the sense of the mass transportation distance of order $p$. However, \eqref{cond(c)} provides also a rate for that convergence. 
\begin{theorem}
\label{t:sufficient_cond}
Assume that for all $j\in J$ the following locally Lipschitz condition is satisfied for all function $f \in \mathfrak{F}_j$. Constants $C_j>0$ exists such that 
  \begin{equation}
  \label{f_j-Lipschitz}
  |f(\eta, x+y)- f(\eta, x) |\leq \bar{C}_j(x,y)\|y\| \quad \text{for all } x\in\mathcal{X}.
  \end{equation}
  where the constant $\bar{C}_j(x,y)\leq C_j\max(1,\|y\|^{p-1},\|x\|^{p-1})$ with $p\geq 1$ and $\mathbb{E} \big[\|x\|^{2(p-1)}\big]$ is finite.
 If the sequence of measures $\{\mu_n\}$ forms a strong approximate identity of order $p$, 
 then equations \eqref{eq.15} and \eqref{eq.16} are satisfied.
\end{theorem}
\begin{proof}
First, we shall show that \eqref{eq.16} holds. 
Using the growth condition and the Jensen's inequality, we obtain
\begin{equation}
\label{eq.26}
    \begin{aligned}
     \sup\limits_{f \in \mathfrak{F}_j} & \sqrt{n} \left|
      \mathbb{E}   \int_{\mathbb{R}^m} f ( \eta, X + y) -
      f( \eta, X) \; \mu_n (dy)   \right|  
      \\
      \leq \sup\limits_{f \in \mathfrak{F}_j} & \sqrt{n}\; 
       \mathbb{E} \int_{\Rb^m}   C_j\max(1,\|y\|^{p-1},\|x\|^{p-1})\|y\| \; \mu_n (dy)  
      \\
      & \leq  \sqrt{n} C_j \Eb 
      \max\left(\int_{\Rb^m} \|y\|\; \mu_n (dy),\int_{\Rb^m} \|y\|^p \; \mu_n (dy) ,\|x\|^{p-1}\int_{\Rb^m} \|y\| \; \mu_n (dy)   \right)   \\
      &\leq \sqrt{n} C_j  \Eb 
      \left(\int_{\Rb^m} \|y\|\; \mu_n (dy)+ \int_{\Rb^m} \|y\|^p \; \mu_n (dy) + \|x\|^{p-1}\int_{\Rb^m} \|y\| \; \mu_n (dy)   \right)\\
      &\leq \sqrt{n} C_j 
      \left(\int_{\Rb^m} \|y\|\; \mu_n (dy)+ \int_{\Rb^m} \|y\|^p \; \mu_n (dy) + \Eb[\|x\|^{p-1}]\int_{\Rb^m} \|y\| \; \mu_n (dy)   \right)
    \end{aligned}
\end{equation}

The term on the right-hand side of equation \eqref{eq.26} converges to zero due to
\eqref{cond(c)}, which proves \eqref{eq.16}. 
To prove \eqref{eq.15}, we proceed in a similar way. 
\begin{equation}
\label{eq.30}
    \begin{aligned}
      &\sup_{f \in \mathfrak{F}_j}   \mathbb{E} 
      \left( \int_{\mathbb{R}^m} f ( \eta, X + y) - f (
      \eta, X)\; \mu_n (dy) \right)^2 \\
      &\leq \mathbb{E}  \left( \int_{\Rb^m} C_j\max(1,\|y\|^{p-1},\|x\|^{p-1})\|y\| \; \mu_n (dy)  \right)^2  \\
      &\leq C_j^2 
      \mathbb{E}  \Big(\int\limits_{\Rb^m} \|y\|\; \mu_n (dy)+\int_{\Rb^m} \|y\|^p \; \mu_n (dy) + \|x\|^{p-1}\int\limits_{\Rb^m} \|y\| \; \mu_n (dy) \Big)^2  \\
      &\leq 3C_j^2 \mathbb{E} \Big[
     \Big(\int_{\Rb^m} \|y\|\; \mu_n (dy)\Big)^2+ \Big(\int_{\Rb^m} \|y\|^p \; \mu_n (dy)\Big)^2 + \Big(\|x\|^{p-1} \int\limits_{\Rb^m} \|y\| \; \mu_n (dy) \Big)^2\Big]\\
     &\leq 3C_j^2 \bigg[
     \Big(\int_{\Rb^m} \|y\|\; \mu_n (dy)\Big)^2+ \Big(\int_{\Rb^m} \|y\|^p \; \mu_n (dy)\Big)^2 + \mathbb{E} \big[\|x\|^{2(p-1)}\big] \Big(\int\limits_{\Rb^m} \|y\| \; \mu_n (dy) \Big)^2\bigg].
    \end{aligned}
\end{equation}
    We have already proved that the maximum converges to zero. Since $t\to t^2$ is continuous, we infer that the right-hand side of  \eqref{eq.30} converges to zero. Consequently,  (\ref{eq.15}) holds as well.
\end{proof}

We obtain more handy conditions when we use the usually kernels smoothing in stochastic optimization. We can state the following result. 
\begin{theorem}
\label{kernel_clt}
Suppose conditions (a3) and (a5) of Theorem \ref{smoothclt} as well as the locally Lipschitz condition \eqref{f_j-Lipschitz} are satisfied. Let a symmetric kernel $K$ with finite $p$ moment be given such that (k1)-(k2) as well as the following conditions are satisfied: 
        \begin{itemize}
            \item[(s1a)]  The integrals $\int \| F_j(x+h_n z)\|_{2,P} K(z) dz $  are finite for all $j\in J$ and for any $x\in \Xc$.%
            \item[(s2b)] The smoothing parameter satisfies $\lim_{n\to\infty} n h_n^2=0.$
        \end{itemize}
Then the result of Theorem~\ref{smoothclt} holds when $d\mu_n(x) = \frac{1}{h_n^m} K\Big(\frac{x}{h_n}\Big)\,dx$.
\end{theorem}
\begin{proof}
    We substitute $\mu_n(dz)$ by $\frac{1}{h_n^{m}}K(\frac{z}{h_n})dz$ and augment the function $h_K^{(n,J)}(\eta)$ defined in the proof of Theorem~\ref{smoothclt} accordingly. 
To verify the assumptions of Theorem~\ref{smoothclt} , we notice that 
assumption (s1a) and the definition of the kernel imply (a1) and (a2). We only have to verify conditions (a4). 
Using Theorem~\ref{t:sufficient_cond}, it is sufficient to verify conditions
\[
\sqrt{n}\int_{\Rb^m} \|h_nz\| K(z)\;dz + \sqrt{n} \int_{\Rb^m} \|h_nz\|^p K(z)\;dz \xrightarrow[n\to\infty]{} 0. 
\]
We see that 
\begin{multline*}
\sqrt{n}\int_{\Rb^m} \|h_nz\| K(z)\;dz + \sqrt{n} \int_{\Rb^m} \|h_nz\|^p K(z)\;dz \\
= \sqrt{n}h_n\int_{\Rb^m} \|z\| K(z)\;dz + \sqrt{n}h_n^p \int_{\Rb^m} \|z\|^p K(z)\;dz     
\end{multline*}
Since the kernel has finite $p$ moment, $p\geq 1$, assumption (s2b) implies the desired convergence. 
Therefore, Theorem \ref{smoothclt} applicable and yields the result. 
\end{proof}

Notice that the smooth estimator and kernel estimator have the same asymptotic distribution.

Popular measures of risk are the mean semi-deviation measures.  
We shall verify (a3) in the Theorem~\ref{smoothclt} and the local Lipschitz condition \eqref{f_j-Lipschitz} in Theorem \ref{t:sufficient_cond}.
Consider the case when small values are preferred, e.g. the random variable represents losses.
\begin{equation*}
\varrho[X] = \Eb[X]+\kappa \| {(\Eb[X]-X)}_{+} \|_{p} = \Eb[X]+\kappa [ \Eb[ (\max (\Eb[X]-X,0) )^p]]^{\frac{1}{p}}
\end{equation*}
where $\kappa \in [0,1]$ and $p >1$.
In this case,
\begin{align*}
f_1(\eta_1,x) = x+\kappa {\eta_1}^{\frac{1}{p}}, \quad
 f_2(\eta_2,x) = [\max(0,\eta_2-x)]^p,\quad
 f_3(x) = x.
\end{align*}
The function $f_1$ satisfies the (a3) if we assume there is an upper bound of set $I$, which is a compact set
\[ | f_1 (\eta, x) - f_1 (\eta', x) | = \kappa |
     \eta^{\frac{1}{p}} - \eta'^{\frac{1}{p}}| \leq 2\kappa \max (|\eta|^{\frac{1}{p}}, |\eta'|^{\frac{1}{p}}) \]
The functions $f_3$  evidently satisfy the local Lipschitz condition and also meet the requirements of condition (a3) when $X$ has sufficiently high moments. We only need to analyze the function $f_2$. When $p=1$, the function has a global Lipschitz constant of 1 with respect to both variables. Consider the case of $p>1$.
Then the function $f_2(\eta_2, \cdot )$ is continuously differentiable, and we obtain the following.
\begin{align*} 
    |f_2(\eta_2, x+y) & - f_2(\eta_2, x)|  = \left|\max\big(0, \eta_2 - (x+y)\big)^p - \max\big(0, \eta_2 -  x\big)^p\right| \\
    &\leq p \max \big(\|\eta_2-x-y\|^{p-1},\|\eta_2-x\|^{p-1}\big)\|y\|\\
    &\leq p c \max \big(1,\|x+y\|^{p-1},\|x\|^{p-1}\big)\|y\|\\
    &\leq pc \max \big(1,(\|x\|+\|y\|)^{p-1}\big)\|y\|
    \leq 2^{p}p c \max \big(1,\|x\|^{p-1},\|y\|^{p-1}\big)\|y\|.
\end{align*}
Here, $c$ is a constant associated with the compact set $I_2$.  Similarly, we can verify (a3) for $f_2$.  We have 
\begin{equation*}
    \begin{aligned}
    &|f_2(\eta, x) - f_2( \eta', x)|= \left|\max\big(0, \eta - x\big)^p - \max\big(0, \eta' -  x\big)^p\right| \\
    &\leq p \max(|\eta- x|^{p-1},|\eta'- x|^{p-1}) \|\eta-\eta'\|
    \leq p \big(\tilde{c}+\|x\|\big)^{p-1})\|\eta-\eta'\|\\
    \end{aligned}
\end{equation*}
Here $\tilde{c}$ is the diameter of the set $I_2.$ When the random variable $X$ has sufficiently high moments, the function $x\to p \big(\tilde{c}+\|x\|\big)^{p-1})$ would have finite $\Lc_2$ norm and  assumption (a3) will be satisfied as well.

\section{Central limit theorem for vector-valued composite functionals}
\label{s:MCLT}

We proceed to establish a more general form of the central limit theorem, extending its applicability to situations involving aggregation of risk measure, which is necessary in evaluation of systemic risk for distributed systems. Let us assume that we deal with a system of $\ell$ agents (components). We consider the random vector $X\in\mathcal{L}_p(\Omega,\mathcal{F},P;\mathbb{R}^\ell)$ comprising the random losses of the agents, i.e.,  the component $X_i$ represents the random loss of agent $i$. When assessing the total risk of the system, two primary approaches are known in literature. The first approach involves selecting a univariate risk measure $\varrho$ and applying it to a function $\Lambda_1: \Rb^\ell\to \Rb$, which aggregates the outcomes of the agents.  
Our results are applicable to this case, as we only need to include the aggregation function $\Lambda_1$ in the composition. 
This approach to systemic risk is applicable when not proprietary information or privacy concerns for the individual agents are present.  The second approach evaluates the total risk by recording the risk evaluation of the individual agents and subsequently aggregating the obtained values. This approach requires handling multiple risk measures at once and estimating their aggregation. 
The goal of this section is to address this situation. 

Let $\Omega_\ell=\{ 1,\cdots,\ell\}$, and the probability space $(\Omega_\ell,\mathcal{F}_c,c)$ is defined by using a vector $c\in\Rb^m_+$ such that $\sum_{i=1}^\ell c_i=1$  as a probability mass function on $\Omega_\ell$ and $\mathcal{F}_c$ as the collection of all subsets of $\Omega_\ell$. We assume that a set of $\ell$ univariate risk measures $\varrho_i:(\Omega,\mathcal{F},P)\rightarrow\Rb$ are used to evaluate the risk of each agent (or system component). There are several ways to aggregate the risk evaluations $\varrho_i(X_i)$, $i=1,\dots,\ell$. We could use an aggregation function $\Lambda_2:\Rb^\ell\to \Rb$ to this end, where the simplest aggregation would be a linear scalarization, i.e., the total risk 
would be given by 
\[
\Lambda_2(\varrho_1[X_1],\dots, \varrho_\ell[X_\ell])= \sum_{i=1}^\ell c_i\varrho_i[X_i]
\]
for the vector $c$ in the simplex of $\Rb^\ell$. In financial literature, it is assumed that $\Lambda_2$ is non-linear function with monotonicity, possibly convexity, and other properties. 
Alternatively, we may apply a more complex aggregation as follows. We define the random variable $X_R$ on the space $\Omega_\ell$ be setting:
\begin{equation*}
    X_R(i)=\varrho_i[X_i],\qquad i=1,\cdots,\ell.
\end{equation*}
Choosing a scalar measure of risk $\rho_0:(\Omega_\ell,\mathcal{F}_c,c)\rightarrow\mathbb{R}$, we define the measure of the total risk (systemic risk) $\rho_{\rm sys}:\mathcal{L}_p(\Omega_\ell,\mathcal{F}_c,c)\rightarrow\mathbb{R}$ as follows:
\begin{equation*}
\varrho_{\rm sys}[X]=\varrho_0[X_R],
\end{equation*}
This measure was proposed in \cite{almen2023risk}, where it was established that it satisfies postulated axioms for systemic measures.
Notice that $\varrho_{\rm sys}[X]$ becomes equal to $\Lambda(\varrho_1[X_1],\dots, \varrho_\ell[X_\ell])$ when $\rho_0(X_R) = \Eb [X_R]$ on the probability space $(\Omega_\ell,\mathcal{F}_c,c).$
Hence, $\varrho_{\rm sys}$ can represent both linear and nonlinear aggregations of $X_R$. To address the statistical estimation of the systemic risk, we shall establish a Central Limit Theorem for vector-valued composite functionals. We shall show how it applies to the systemic risk estimation in due course. 

For the sake of generality, we shall consider the case of $\ell$ random vectors $X^1,\dots X^\ell$ with realization in $\Rb^{s_{i}}$, $i=1,\dots \ell$, which may be dependent. We define a random vector $X \in \mathbb{R}^{m},$ where $m=s_1+\dots +s_\ell$ containing all vectors $X^i$ stacked, i.e.,  $X = (X^1, \dots ,X^\ell)$ and $m_{0}^{i}=1$.
We define $\ell$ composite risk functionals as follows:
\begin{equation} 
 \varrho_{i}[X^{i}] = \mathbb{E} \left[ f_{1}^{i}\left( \mathbb{E} \left[  f_{2}^{i}\left( \mathbb{E} \left[ \cdots f_{k_{i}}^{i}\left(\mathbb{E}\left[ f_{k_{i}+1}\left( X^{i} \right)\right],X^{i}\right)\right]\cdots,X^{i} \right)\right],X^{i}\right)\right] 
\end{equation}
where for $i = 1\ldots \ell$. For $j = 1\ldots k_{i}$,  we define $f_{k_{i}+1}^{i}(X^{i}):\mathbb{R}^{s_{i}} \rightarrow \mathbb{R}^{m^{i}_{k}}$ ,  and $f_{j}^{i}(\eta_{j}^{i},X^{i}): \mathbb{R}^{m_{j}^{i}}\times\mathbb{R}^{s_{i}}\rightarrow \mathbb{R}^{m_{j-1}^{i}}$.
Similarly to the univariate case, we define for $i = 1 \ldots \ell$ the following:
\begin{align*}
\bar{\eta}_{k+1}^{i} & = \Eb[f_{k+1}^{i}(X^{i})]\\
\bar{f}_{j}^{i}(\eta_{j}) & = \Eb[f_{j}^{i}(\eta_{j},X^{i})],\quad \bar{\eta}_{j}^{i} = \bar{f}_{j}^{i}(\bar{\eta}_{j+1}^i) \quad j = 1\ldots k_{i} 
\end{align*}
Assume for a moment a common nesting order $k$, we shall show later that this can always be achieved with no loss of generality. We define functions $g_{j}$ in the following way
\[
g_{k+1}(X)=\begin{pmatrix}
f_{k+1}^1(X^1) \\
\vdots \\
f_{k+1}^\ell(X^\ell)
\end{pmatrix}  
\quad
g_{j}(\eta_{j},X)=\begin{pmatrix}
f_{j}^{1}(\eta_{j}^1,X^1) \\
\vdots \\
f_{j}^{\ell}(\eta_{j}^\ell,X^\ell) 
\end{pmatrix}
\quad j = 1\ldots k_{i}.
\]
So that, $g_{k+1}(X):\mathbb{R}^{m}\rightarrow \mathbb{R}^{m_{k}}$ where $m_{k}=m^{1}_{k}+\dots+m^{\ell}_{k}$ and
$g_{j}(\eta_{j},X):\mathbb{R}^{m_{j}}\times\mathbb{R}^{m}\rightarrow \mathbb{R}^{m_{j-1}}$  where $m_{j}=m_{j}^{1}+\ldots+m_{j}^{\ell}$.
Then we have the multiple composite risk functional
\begin{equation}
\label{d:rho-vector} 
\begin{gathered}
 \varrho[X] = \mathbb{E} \left[ g_{1}\left( \mathbb{E} \left[  g_{2}\left( \mathbb{E} \left[ \cdots g_{k}\left(\mathbb{E}\left[ g_{k+1}\left( X \right)\right],X\right)\right]\cdots,X \right)\right],X\right)\right]. \\
 \text{and }\; \varrho[X]= (\varrho_{1}(X^{1}),\dots,\rho_{\ell}(X^{\ell})). 
 \end{gathered}
\end{equation}
Note that $g_{1}(\eta_{1},X): \mathbb{R}^{m_{1}}\times\mathbb{R}^{m}\rightarrow \mathbb{R}^{\ell}$. 
The following quantities become relevant
\begin{align*}
\bar{\eta}_{k+1} & = \Eb[g_{k+1}(X)] \equiv [\bar{\eta}^{1}_{k+1}(u),\ldots,\bar{\eta}^{\ell}_{k+1}(u)]\\
\bar{g}_{j}(\eta_{j}) & = \Eb[g_{j}(\eta_{j},X)] = (\bar{f}_{j}(\eta_{j}^{1}),\ldots,\bar{f}_{j}(\eta_{j}^{\ell}))]\\
\bar{\eta}_{j} & = \bar{g}_{j}(\bar{\eta}_{j+1}) \quad j = 1\ldots k.
\end{align*}
Hence $\varrho [X] = \bar{\eta}_{1} \in \mathbb{R}^{\ell}$ by construction.
We redefine the collection of functions $\Ff$ to include all the functions involved.
\begin{equation}
\label{d:Ff-redef}
\begin{aligned}
\mathfrak{F}_j\,&= \Big\{ g_{j,i}(\eta_j,\cdot): \Xc\rightarrow \Rb, \;
 i=1,\dots m_{j-1},\;\; \eta_j\in I_j,\; u\in \Uc,  \Big\}  \quad  j=1,\dots, k; \\
\mathfrak{F}_{k+1}\;& = \Big\{ g_{k+1,i}(\cdot): \Xc\rightarrow \Rb, \;\; i=1,\dots,m_{k+1},\; u\in \Uc  \Big\}, 
\quad \mathfrak{F} =\cup_{1\leq j\leq k+1} \mathfrak{F}_j. 
\end{aligned}
\end{equation}

We now state the central limit theorem of mixed smoothed and empirical estimators for the scalar-valued composite risk functional defined as above.

\begin{theorem}\label{multivariateclt}
Suppose a sequence of smoothing measures $\{ \mu_n \}$ on $\Rb^m$ is given and an index set $J$ is fixed. If the conditions of Theorem~\ref{smoothclt} are satisfied for $\{ \mu_n \}$, for all functions involved in the definition of \eqref{d:rho-vector}, and for the envelope functions of the classes \eqref{d:Ff-redef}, then  
\[
\sqrt{n}\big(\varrho^{(n,J)}_\mu [X]- \varrho[X]\big) = \sqrt{n}\left[ \begin{pmatrix} \rho_{\mu,1}^{(n,J)}\\ \vdots \\ \rho_{\mu,\ell}^{(n,J)} \end{pmatrix}  - \begin{pmatrix} \rho_1\\ \vdots \\ \rho_\ell\end{pmatrix} \right]\dto \xi_{1} (W),
\]
where $W(\cdot) = \big( W_{1}(\cdot),\dots,W(\cdot)_{k+1} \big)$ is a zero-mean Brownian bridge with $W_{1}$ being an $m$-dimensional zero-mean Brownian bridge, $W_{j}$ is a zero mean Brownian bridge of dimension $m_{j-1}$, and 
\begin{equation}
\label{defxi}
\begin{aligned}
&\xi_{k+1}(d)  = d_{k+1}\\
&\xi_{j}(d)  = \int_{\chi}g'_{j}(\bar{\eta}_{j+1},x;\xi_{j+1}(d))P(dx) + d_{j}(\bar{\eta}_{j+1}) \text{     } j = k,k-1, \ldots,1.
\end{aligned}
\end{equation} 
The covariance function is given by
\begin{align*}
\cov\left[W_{i}(\eta_{i}),W_{j}(\eta_{j})\right] & = \int_{\chi}\left[g_{i}(\eta_{i},x)-\bar{g}_{i}(\eta_{i})\right]\left[g_{j}(\eta_{j},x)-\bar{g}_{j}(\eta_{j})\right]^\top P(dx)\\
\cov\left[W_{i}(\eta_{i}),W_{k+1}(u)\right] & = \int_{\chi}\left[g_{i}(\eta_{i},x)-\bar{g}_{i}(\eta_{i})\right]\left[g_{k+1}(x)-\bar{\eta}_{k+1}\right]^\top P(dx)\\
\cov\left[W_{k+1}(u),W_{k+1}(u)\right] & = \int_{\chi}\left[g_{k+1}(x)-\bar{\eta}_{k+1}\right]\left[g_{k+1}(x)-\bar{\eta}_{k+1}\right]^\top P(dx)
\end{align*}
Additionally, if all functions are differentiable for all values of their last argument, then $\xi_1 (W)$ has the normal distribution $N (0, C^\top  \Sigma_g C),$
where $\Sigma_g$ is the covariance of $W$, and 
\[
C^\top  = \big( 
I_{\ell\times\ell} \; C_{1} \; C_{2} \; \ldots \; C_{k} 
\big), \quad C^\top _r = \left(
 \prod_{j=1}^r \Eb [g_j' (\bar{\eta}_{j + 1})]
  \right), \quad r=1,\ldots,k,
  \]
  Here $g_j' (\bar{\eta}_{j + 1})$ denotes the Jacobian of $g$ with respect to the first argument calculated at $\bar{\eta}_{j + 1}.$ 
\end{theorem} 
\begin{proof}
Each functional $\varrho_{i}$ is a composition of functions $f^i_j$, $j=1,\dots k_i$ and $i=1,\dots \ell$. We note that the value of $\varrho_{i}$ does not change if the composition of functions continued beyond $f_{1}^{i}$ by composition with the identity function. 
One can compose with the identity any number of times without changing the form or value of the resulting risk measure. Hence, without loss of generality, we may assume that $k_{i} = k$ is a common number for all risk functionals $\varrho_i$, $i=1,\dots \ell$. Otherwise we may set  $k = \max\{ k_i: i=1,\dots,\ell\}$ and relabel the indices, redefining the functions as $\tilde{f}^{i} = [\eta_{1},\ldots,\eta_{k-k_{i}},f_{1}^{i},\ldots,f_{k_{i}+1}^{i}]$ so that $\tilde{f}_{j}^i$ is the identity function for all $j = 1,\ldots, k-k_{i}$, $i = 1,\dots,\ell$. 
The statement (i) follows immediately from Theorem \ref{smoothclt}. 

Now consider the differentiable case. In the construction of $\xi_{k + 1} (d)$, we obtain
    \begin{equation*}
      \xi_k (d) = \mathbb{E} [g_k' (\bar{\eta}_{k + 1}) \cdot \bar{\eta}_{k+1}(d)] + d_k(\bar{\eta}_{k + 1}) =
     \mathbb{E} [g_k' (\bar{\eta}_{k + 1})] \cdot d_{k + 1} + d_k (\bar{\eta}_{k + 1})
    \end{equation*}
Proceeding the same way, we get 
   \begin{align*}
      \xi_{k - 1} (d) = & \mathbb{E} [g_{k - 1}' (\bar{\eta}_k) \cdot d_k] + d_{k -1} (\bar{\eta}_k)\\
      = & \mathbb{E} [g_{k - 1}' (\bar{\eta}_k) \cdot (\mathbb{E} [g_k' (\bar{\eta}_{k + 1})\cdot d_{k + 1}] + d_k (\bar{\eta}_{k + 1}))] + d_{k - 1} (\bar{\eta}_k) \\
      = & \mathbb{E} [g_{k - 1}' (\bar{\eta}_k)] \cdot \mathbb{E} [f_k' (\bar{\eta}_{k +1})] \cdot d_{k + 1} +\mathbb{E} [f_{k - 1}' (\bar{\eta}_k)] \cdot d_k (\bar{\eta}_{k+ 1}) + d_{k - 1} (\bar{\eta}_k) 
    \end{align*}
Substituting the random vector $W$ in place of the direction $d$ we obtain the following expression for the limiting distribution,
    \begin{equation}\label{check_con}
       \xi_1 (W) =  C^\top _k \cdot W_{k + 1} + C^\top _{k - 1} \cdot W_k (\bar{\eta}_{k + 1}) + \cdots
      + C^\top _1 \cdot W_2 (\bar{\eta}_3) + W_1 (\bar{\eta}_2) 
    \end{equation}
    where the matrices  $C_r$ are defined as stated. 
    Relation \eqref{check_con} implies $\xi_1 (W) = C^\top  W$ so that the variance is
    given by     ${\Var} [\xi_1 (W)] = {\Var}[C^\top  W] = C^\top  \Sigma_g C.$ 
\end{proof}

We observe that if the functions $f_j$ with $j\in J$ satisfy the Lipschitz conditions \eqref{f_j-Lipschitz}, then condition (b2) is satisfied. 

\begin{corollary}\label{kernel_mclt}
Assume the conditions of Theorem~\ref{kernel_clt} for all functions involved in the definition of \eqref{d:rho-vector}, and for the envelope functions of the classes \eqref{d:Ff-redef}. 
Given a symmetric kernel $K$ satisfying (k1)-(k2) as well as (s1a)-(s2b), 
the result of Theorem \ref{multivariateclt} holds when $d\mu_n(x) = \frac{1}{h_n^m} K\Big(\frac{x}{h_n}\Big)\,dx$.
\end{corollary}
\begin{proof}
    The statement follows by combining the arguments of the proofs of Theorem~\ref{kernel_clt} and Theorem~\ref{multivariateclt}.

\end{proof}

\subsection{Central limit theorem for sample-based stochastic optimization}
\label{sec:decision} 

In this section, we return to the kernel-smoothed estimator of the optimal value of a composite optimization problem. Recall the formulation of the associated composite optimization problems. 
\begin{equation}
\label{p:last}
    \vartheta= \min_{u \in U} \mathbb{E}[f_1(u,\mathbb{E}[f_2(u,\mathbb{E}[\ldots f_k(u,\mathbb{E}[f_{k+1}(u,X)],X)] 
    \ldots,X)],X)]
\end{equation}
where $X$ is a {m}-dimensional random vector; $U$ is a closed convex set in $\Rb^n$ and  the function $f_j: \Rb^n \times \mathbb{R}^{m_j} \times \mathbb{R}^{m} \rightarrow \mathbb{R}^{m_{j-1}}$, $j=1,\cdots,k$ with $m_0=1$, and $f_{k+1}: \Rb^n \times \mathbb{R}^{m} \rightarrow \mathbb{R}^{m_k}$ are continuous. We shall assume that problem \eqref{p:last} is solvable and has a unique optimal solution denote $\hat{u}$.  
We fix a sufficiently large compact set $\Uc$ such that $\hat{u}\in \intt \Uc\subset U.$
Further, we again fix compact sets $I_1\subset\Rb^{m_1},\cdots,I_k\subset\Rb^{m_k}$ such that $\bar{f}_{j+1}(\Uc,I_{j+1}) \subset \intt(I_j)$, $j=1,\cdots,k-1$, and $\bar{f}_{k+1}(\Uc) \subset \intt(I_k)$, where $\intt(I_j)$ stands for the interior of $I_j.$ Without loss of generality, we assume that $\Uc$ and $I_j$, $j=1,\dots $ are convex sets.

Our first goal is to show  a central limit formulae for the optimal value of problem \eqref{p:last} which used mixed smooth and empirical estimators and in a second step, we shall specialize the statement when the smoothing uses the same kernel. 
We shall follow a similar line of proofs with slightly modified definitions of the auxiliary objects. Analogous functions are defined as:
\begin{equation}
\label{d:Ff-(u,eta)}
\begin{split}
    & \bar{f}_j(u,\eta_j)=\int_{\mathcal{X}}f_j(u,\eta_j,x)\;P(dx), \ \ \ \ j=1,\cdots,k \\
    &  \bar{f}_{k+1}(u)=\int_{\mathcal{X}}f_{k+1}(u,x)\;P(dx) \\
   & \bar{\eta}_{k+1}=\bar{f}_{k+1}(\hat{u}),\quad \bar{\eta}_{j}=\bar{f}_{j}(\hat{u},\hat{\eta}_{j+1})\quad j=1,\cdots,k. 
\end{split}
\end{equation}
The vector function $f(u, \eta,x)$ is defined on $\Uc \times I \times \mathcal{X} \rightarrow \mathbb{R}^{m}$ be setting 
\[
f(u,\eta,x)=[f_1(u,\eta_1,x),f_2(u,\eta_2,x),\cdots,f_k(u,\eta_k,x),f_{k+1}(x)]^{\top}.
\] 
We modify the classes of functions as follows:
\begin{align*}
\mathfrak{F}_j\,&= \Big\{  f_{j,i}(u,\eta_j,\cdot): \Xc\rightarrow \Rb, \;
 i=1,\dots m_{j-1},\;\; \eta_j\in I_j,\; u\in \Uc,  \Big\}  \quad  j=1,\dots, k; \\
\mathfrak{F}_{k+1}\;& = \Big\{ f_{k+1,i}(u,\cdot): \Xc\rightarrow \Rb, \;\; i=1,\dots,{m_k},\; u\in \Uc  \Big\}, 
\quad \mathfrak{F} =\cup_{1\leq j\leq k+1} \mathfrak{F}_j. 
\end{align*}
The associated envelope functions is given by
\[
F_j(x)= \underset{\eta_j \in I_j, u \in \Uc}{\sup} | f_{j,i} (u, \eta_j, x) |,\quad F_{k+1}(x)= \underset{u \in \Uc}{\sup} | f_{k+1,i} (u,x) |
\]
Due to the compactness of $\Uc\times I_j$, the functions $F_j(\cdot)$, $j\in\Jc$ are well-defined and measurable (\cite[Theorem 7.42]{mainbook}). The space $\Hc$ is now defined as follows:
\begin{equation*}
    \mathcal{H}={\mathcal{C}}_1^{(0,1)}(\Uc \times I_1) \times {\mathcal{C}}_{m_1}^{(0,1)}(\Uc \times I_2) \times \cdots \times {\mathcal{C}}_{m_{k-1}}^{(0,1)}(\Uc \times I_k) \times {\mathcal{C}}_{m_k}(\Uc)
\end{equation*}
where  ${\mathcal{C}}^{(0,1)}_{m_{j-1}}$ is the space of ${\mathbb{R}}^{m_{j-1}}$-valued continuous function on $\Uc \times I_j$, which is continuous with respect to the first component and continuously differentiable with respect to the second component. We denote the Jacobian of $f_j(u,\eta_j,x)$ with respect to the second argument calculated at $\eta_{j}^{\ast}\in I_j$ by $f_{j}^{'}(u,\eta^{\ast},x)$. For every direction $d \in \mathcal{H}$, we define recursively the following elements:
\begin{equation}
\label{recursive2}
\begin{gathered}
\xi_{k+1}(d) = d_{k+1},\\
\xi_{j}(d) = \int_{\Xc}  f'_{j}(\hat{u},\bar{\eta}_{j+1},x)\xi_{j+1}(d)\,P(dx) + d_{j}\big(\bar{\eta}_{j+1}\big),\quad j=k,k-1,\dots,1.
\end{gathered}
\end{equation}
Suppose an index set $J\subseteq\{1, \ldots, k+1\}$ is fixed to determine the layers at which we shall apply smoothing.      

\begin{theorem} 
\label{CLTsmooth}
Suppose assumptions (a1)-(a2) of Theorem~\ref{smoothclt} is satisfied and additionally the following conditions hold.
\begin{itemize}
    \item[(b1)] Functions $\tilde{F}_j:\Xc\to \Rb$, $j\in\Jc$ with a finite $\Lc_2 (P)$ norm exists such that for all $(u,\eta),(u',\eta')\in\Uc\times I_j$ and all $x\in\Xc$
\begin{gather*}
\|f_j(u,\eta,x)-f_j(u',\eta',x)\| \leq \|(u,\eta) -(u',\eta')\|\tilde{F}_j(x),\quad j=1,\dots, k+1.
\end{gather*}
    \item[(b2)]
    The following two conditions are satisfied for all $j\in J$:
    \begin{gather}\label{eq.15a}
      \sup\limits_{f \in \mathfrak{F}_j}\; \mathbb{E} \Big[ \Big(
      \int_{\mathbb{R}^m} \big(f (u,\eta, X + y) - f(u,\eta, X)\big) \;
      \mu_n (dy) \Big)^2 \Big]  \xrightarrow[n \rightarrow \infty]{} 0;
      \\
    \label{eq.16a}
      \underset{f \in \mathfrak{F}_j}{\sup} \; \sqrt{n} 
      \Big| \mathbb{E} \Big[ \int_{\mathbb{R}^m} \big( f (u,\eta, X + y) -f (u,\eta, X) \big) \;\mu_n (dy)  \Big] \Big| \xrightarrow[n \rightarrow \infty]{} 0.
    \end{gather}
    \item[(b3)] The functions $f_j(u,\cdot,x)$, $j=1,\cdots,k$ are continuously differentiable for every $x \in \mathcal{X}, u \in U$. Moreover, their derivatives $\nabla_\eta f_j(u,\eta,x)$ are continuous with respect to the first two arguments and $\nabla_\eta f_j(u,\eta,x+y)$ is uniformly bounded by a $\mu_n$-integrable function $g(y).$
  \end{itemize}
Then 
$\sqrt[]{n}\big[ \vartheta^{(n,J)}_{\mu}-\vartheta\big] \dto \xi_1(G)$ 
where $G(\cdot)=(G_1(\cdot),\ldots,G_k(\cdot),G_{k+1})$ is zero-mean Brownian process on $I =I_1 \times I_2 \times \ldots \times I_k$. The covariance function of $G$ has the following form:
\begin{equation} 
\begin{aligned}
 \cov[G_i(\eta_i),G_j(\eta_j)] & =\int_{\mathcal{X}}[f_i(\hat{u},\eta_i,x)-\bar{f}_i(\hat{u},\eta_i)][f_j(\hat{u},\eta_j,x)-\bar{f}_j(\hat{u},\eta_j)]^{\top}\; P(dx) \\
 & \qquad\qquad\qquad\eta_i \in I_i, i=1,\ldots,k, \\
	\cov[G_i(\eta_i),G_{k+1}] & =\int_{\mathcal{X}}[f_i(\hat{u},\eta_i,x)-\bar{f}_i(\hat{u},\eta_i)][f_{k+1}(\hat{u},x)-{\bar{f}}_{k+1}(\hat{u})]^{\top} \; P(dx) \\
\cov[G_{k+1},G_{k+1}]& =\int_{\mathcal{X}}[f_{k+1}(\hat{u},x)-{\bar{f}}_{k+1}(\hat{u})][f_{k+1}(x)-{\bar{f}}_{k+1}(\hat{u})]^{\top} \; P(dx)
\end{aligned}   
\end{equation}
\end{theorem}
\begin{proof}
We start from the function $h_\mu^{(n,J)}(u,\eta)$. 
\begin{equation}
\big[h_\mu^{(n, J)}\big]_j(\eta)=
\begin{cases}
\frac{1}{n}\sum_{i=1}^{n}\int f_j(u,\eta,X_i+ z) d\mu_n(z) &\text{ if } j\in J, j=1,\dots, k,\\
\frac{1}{n}\sum_{i=1}^{n} f_j(u,\eta,X_{i}) &\text{ if } j\not\in J, j=1,\dots, k,\\
\frac{1}{n}\sum_{i=1}^{n}\int f_{k+1}(u,X_{i}+ z) d\mu_n(z) &\text{ if } k+1\in J,\\
\frac{1}{n}\sum_{i=1}^{n} f_{k+1}(u,X_{i}) &\text{ if } k+1\not\in J.\\
\end{cases}
\end{equation}
The function $\bar{f}(u,\eta)$ is defined as
$\bar{f}(u,\eta)=\; P f=\int_{\mathcal{X}}f(u,\eta,X)\;dx.$
We observe again that all assumptions $(a1)$-$(a4)$ of Theorem \ref{smoothclt} are satisfied if we consider the pair $(u,\eta)$ instead of $\eta$. Hence, we infer 
\begin{equation*}
    \sqrt[]{n}(h^{(n,J)}_{\mu}-\bar{f}) \dto G
\end{equation*}
where ${G}$ is a standard Brownian process with zero-mean and covariance function 
\begin{equation*} 
\cov[G(u,{\eta}),G(u', \eta')]=\int_{\Xc}\big[f(u, \eta,x)-\bar{f}(u,\eta)\big]\big[f(u',\eta',x)-\bar{f}(u', \eta')\big]^{\top} \;P(dx)
\end{equation*}
We define the operator $\Psi: \Hc \rightarrow \Rb$ as follows
        \begin{equation*}
           \Psi(u,h) = h_1\Big(u,h_2\big(u,\;\cdots h_k(u,h_{k+1}(u))\;\cdots\big)\Big).
        \end{equation*}
        and consider the optimal value of the parametric problem
\[
\psi(h) = \min_{u \in U} {\Psi}(u,h).
\]
When $h=\bar{f}$, we have $\Psi(u,\bar{f})=\bar{f_1}(u, \bar{f_2}(u, \cdots \bar{f_k}(u, f_{k+1})  \cdots))$. When $h=h^{(n,J)}_{\mu}$, we have $\Psi(u,h^{(n,J)}_{\mu})=\vartheta^{(n,J)}_{\mu_n}.$ These facts allow us to use the infinite-dimensional delta theorem  in \cite{delta} to transfer the convergence of $h^{(n)}_{\mu}$ to the convergence of $\Psi(u, h^{(n)}_{\mu})=\vartheta_{\mu}^{(n)}$ provided that we verify the Hadamard directional differentiability of the optimal value function $\psi(\cdot)= \min_{u \in U} \Phi(u, \cdot)$ at $\bar{f}$. To this end, we use the Theorem 4.13 in \cite{shapiro2013}, which  implies that the optimal value function $v(\cdot)$ is Fr\'echet differentiable. Denoting its derivative by ${\Phi}^{'}_h(u,\bar{f})$, we infer the  Hadamard directional differentiability of $\psi(h)$ at $\bar{h}$ in every direction $d$ with $v^{'}(\bar{f};d)=\min_{u \in U} \Phi_{h}^{'}(u,\bar{f})d.$

The remaining part of the proof is the same as the proof of Theorem \ref{smoothclt}.
\end{proof}

Notice that if the functions $f_j$ with $j\in J$ satisfy the Lipschitz conditions \eqref{f_j-Lipschitz}, then condition (b2) is satisfied. 

Consider now smoothing my kernels. 
The following result may be proved in an analogous way as Theorem \ref{CLTsmooth}.

\begin{corollary}\label{CLTkernel}
Assume the conditions of Theorem~\ref{kernel_clt} for the functions involved and for the envelope functions of the classes \eqref{d:Ff-(u,eta)}. 
Given a symmetric kernel $K$ satisfying (k1)-(k2) as well as (s1a)-(s2b), 
the result of Theorem \ref{multivariateclt} holds when $d\mu_n(x) = \frac{1}{h_n^m} K\Big(\frac{x}{h_n}\Big)\,dx$.
\end{corollary}
\begin{proof}
The statement follows by virtue of Theorem\ref{CLTsmooth} after verifying assumption (b2) in the same way as in the proof of Theorem~\ref{kernel_clt}.
\end{proof}

For stochastic problem, we introduce a portfolio optimization problem where the random returns of the securities are represented by a random vector $X$. The portfolio itself is characterized by a vector $u$, which denotes the allocation of the available capital across the securities. The objective is to address the mean-semideviation optimization problem for potential losses. The problem is formulated as follows:
\begin{equation*}
\varrho[u,X]=-\Eb[u^TX]+\kappa \| {(\Eb[u^TX]-u^TX)}_{+} \|_{p}=-\Eb[u^TX]+\kappa [ \Eb[ (\max \{\Eb[u^TX]-u^TX,0\} )^p]]^{\frac{1}{p}}
\end{equation*}
where $\kappa \in [0,1]$ and $p >1$. In this case,
\begin{align*}
f_1(u,\eta_1,x)=-u^Tx+\kappa {\eta_1}^{\frac{1}{p}},
 f_2(u,\eta_2,x)=[\max\{0,\eta_2-u^Tx\}]^p,
 f_3(u,x)=u^Tx
\end{align*}
To verify the locally Lipschitz condition, \( f_1 \) has a modulus of continuity \( w(t) = c_1 t \), where \( c_1 \) is the maximum of the norm of \( u \in \mathcal{U} \). Specifically, we have

\begin{equation*}
    |f_1(u, \eta_1, x+y) - f_1(u, \eta_1, x)| = |u^T y| \leq c_1 \|y\|,
\end{equation*}
where the last inequality follows from the Cauchy-Schwarz inequality. The same verification applies to \( f_3 \) in a similar manner. The remaining verifications follow the same process as the previous example.

\section{Application}
\label{s:simulation}
\subsection{Risk measures representable as optimal value of an optimization problem}\label{5.1}

Now we discuss the special case of two-level composite stochastic optimization
\begin{equation*} \label{CLT-2level}
    \vartheta= \min_{u \in U} f_1(u,E[f_2(u,X)])
\end{equation*}
This structure appears when we evaluate the Average (Conditional) Value-at-Risk or 
higher order measures of risk. 
Recall that higher order risk measure with $c=\frac{1}{\alpha}>1$ is defined as follows (\cite{krokhmal2007higher,Dentcheva2024}:
\begin{equation}
\label{d:inverse}
\rho[X] = \min_{u\in \Rb} \Big\{ u + c \big\|\max(0, X- u)\big\|_p \Big\}
\end{equation}
We represent this measure as a the optimal value of a composite functional by setting:
\begin{equation*}
f_1(u,\eta,x)  = u + c \eta^{\frac{1}{p}},\quad 
f_2(u,x)  = \big(\max(0, x- u)\big)^p.
\end{equation*}
For $p>1$ and $c>1$, the optimization  problem on the right-hand side of \eqref{d:inverse} has a unique solution, which we denote by $\hat{u}$. In that case, we can select a compact set $\Uc$ sufficiently large to contain the point $\hat{u}.$
For $p=1$, we do not need a composition; the problem reduces to the optimization of the expected value of a convex function as well as the case in this paper specialized to one layer.  

Consider the estimators based on a proper approximate identity $\{\mu_n\}$ and on a kernel $K$, respectively. 
\begin{gather*}
   \vartheta_{\mu}^{(n)} = \min_{u\in\Uc}f_1\Big(u,\sum_{i=1}^{n}\frac{1}{n} \int f_2(u,X_i+y)\;\mu_n(dy)\Big) \\
   \vartheta_K^{(n)}=\min_{u\in\Uc} f_1\Big(z,\sum_{i=1}^{n}\frac{1}{n}\int f_2(u,y)K\Big(\frac{y-X_i}{h_n}\Big)\frac{1}{h_n^m}\;dy\Big). 
\end{gather*}
Following the same technique, we define the functions:
    \begin{gather*}
        h_{\mu}^{(n)}(u)=\frac{1}{n}\sum_{i=1}^{n}\int f_2(u,X+y)d \mu_n(y)\\
        h_K^{(n)}(u)=\frac{1}{n}\sum_{i=1}^{n}\int f_2(u,x)\frac{1}{h_n^m}K(\frac{x-X_i}{h_n})dx.
    \end{gather*}
The mapping $\Phi(u,h)=f_1(u,h(u))$, where the operator $\Phi:\Uc \times C(\Uc) \rightarrow \Rb$, and the space $\Rb^n \times C(\Uc)$ is equipped with product norm of the Euclidean norm on $\mathbb{R}^{n}$ and supremum norm of $C(\Uc)$ . The functional $\psi$ is defined $\psi(h)=\min_{u \in \Uc} \Phi(u,h)$, where $\psi: C(\Uc) \rightarrow  \Rb$. It is easy to see that $\vartheta= \psi(\bar{f}_2)$, $\vartheta_{\mu}^{(n)}=\psi(h_{\mu}^{(n)})$ and $\vartheta_K^{(n)}=\psi(h_{K}^{(n)})$. 
We denote the set of optimal solutions in \eqref{CLT-2level} as $\hat{U}$. The class $\Ff$ for this setting is defined as $\Ff=\{ f_2(u, \cdot)\;\; u\in\Uc)\}$.  
\begin{corollary}
\label{c:2level-mu}
Suppose the conditions (a1)-(a4) of Theorem \ref{smoothclt} are satisfied and additionally $f_1(u,\cdot)$ is differentiable for all $u\in \Uc$ , and both $f_1(\cdot,\cdot)$ and its derivative w.r.t. the second argument $\nabla f_1(\cdot,\cdot)$ are continuous w.r.t. both arguments. 
Then 
\begin{equation*}
\sqrt[]{n}[ \vartheta_{\mu}^{(n)}-\vartheta] \dto \min_{u \in \hat{U}}
	\langle \nabla f_1(u,\Eb[f_2(u,X)]),G(u) 	\rangle
\end{equation*}
where $G(u)$ is a zero-mean Brownian process on $G$ with covariance function 
\begin{equation*}
    cov[G(u),G(u')]=\int_{\mathcal{X}}(f_2(u,x)-\Eb[f_2(u,X)])(f_2(u',x)-\mathbb{E}[f_2(u',X)])^{\top} P(dx)
\end{equation*}
Moreover, if the optimal solution set $\hat{U}$ contains only one element $\hat{u}$, then 
\begin{equation*}
\sqrt[]{n}[ \vartheta_{\mu}^{(n)}-\vartheta] \dto 
    \langle \nabla f_1(\hat{u},\mathbb{E}[f_2(\hat{u},X)]),G(\hat{u})   \rangle
\end{equation*}
where $G(\hat{u})$ is a zero-mean normal vector with covariance 
\[
\cov[G(\hat{u}),G(\hat{u})]=\cov[f_2(\hat{u},X),f_2(\hat{u},X)].
\]
\end{corollary}
The proof follows the same line of arguments as the proof of Theorem \ref{CLTsmooth} and it is omitted. 
\begin{corollary}
Suppose the conditions (a1)-(a4) of Theorem \ref{smoothclt} are satisfied,  additionally $f_1(u,\cdot)$ is differentiable for all $u\in \Uc$ , and both $f_1(\cdot,\cdot)$ and its derivative w.r.t. the second argument $\nabla f_1(\cdot,\cdot)$ are continuous w.r.t. both arguments and 
as well as the locally Lipschitz condition \eqref{f_j-Lipschitz} are satisfied for $f_2(u,\cdot)$. Let the symmetric kernel $K$ satisfy (k1)-(k2) and (s1)--(s2). Then the conclusions of Corollary~\ref{c:2level-mu} are satisfied for $\vartheta^{(n)}_K$.
\end{corollary}

In our numerical experiments, we have used the higher-order inverse risk measure as defined in \eqref{d:inverse} with parameters $c=20$ and $p=2$. We take independent identically distributed observations $X_i,i=1,\cdots,n$ from an independent identically distributed $X\sim\mathcal{N}(10,3)$ observations. 
The theoretical minimum is attained at $\hat{u}=14.5048$ resulting in the risk value $\varrho[X]=15.5163$. We consider the uniform kernel estimation $K(u)=\frac{1}{2h_n}$ with support on $|u| \leq h_n$. 
\begin{small}
\begin{multline*}
     \varrho^{(n)}_K=\min_{u \in\mathbb{R}} \Big\{ u+c\Big(\frac{1}{n}\sum_{i=1}^{n}\int (\max(0,y-u))^p \frac{1}{2 h_n} \mathbb{I}_{(|y-X_i| \leq h_n)}\; dy \Big)^{\frac{1}{p}}  \Big\} \\
      = \min_{u \in \mathbb{R}} \Big\{u+c\Big(\sum_{i=1}^{n}\frac{1}{2n(p+1)h_n}\big[(\max(0,h_n+X_i-u))^{p+1} \\
      -(\max(0,-h_n+X_i-u))^{p+1}\big]\Big)^{\frac{1}{p}}\Big\}  
\end{multline*}
\end{small}
The estimators $\varrho^{(n)}_{K,j}, j=1,2,\cdots,m$ for sample size $n$ are evaluated by $m$ replications. which comes
 \begin{equation*}
    \sqrt{n}[\rho^{(n)}_K-\rho] \dto \frac{c}{p}(\Eb\big[(\max(0,X-\hat{u}))\big]^p)^{\frac{1-p}{p}} G.
 \end{equation*}
  where $G$ is a standard normal random variable with zero and variance

\begin{figure}
    \begin{center}
        \includegraphics[width=2.5in,height=1in]{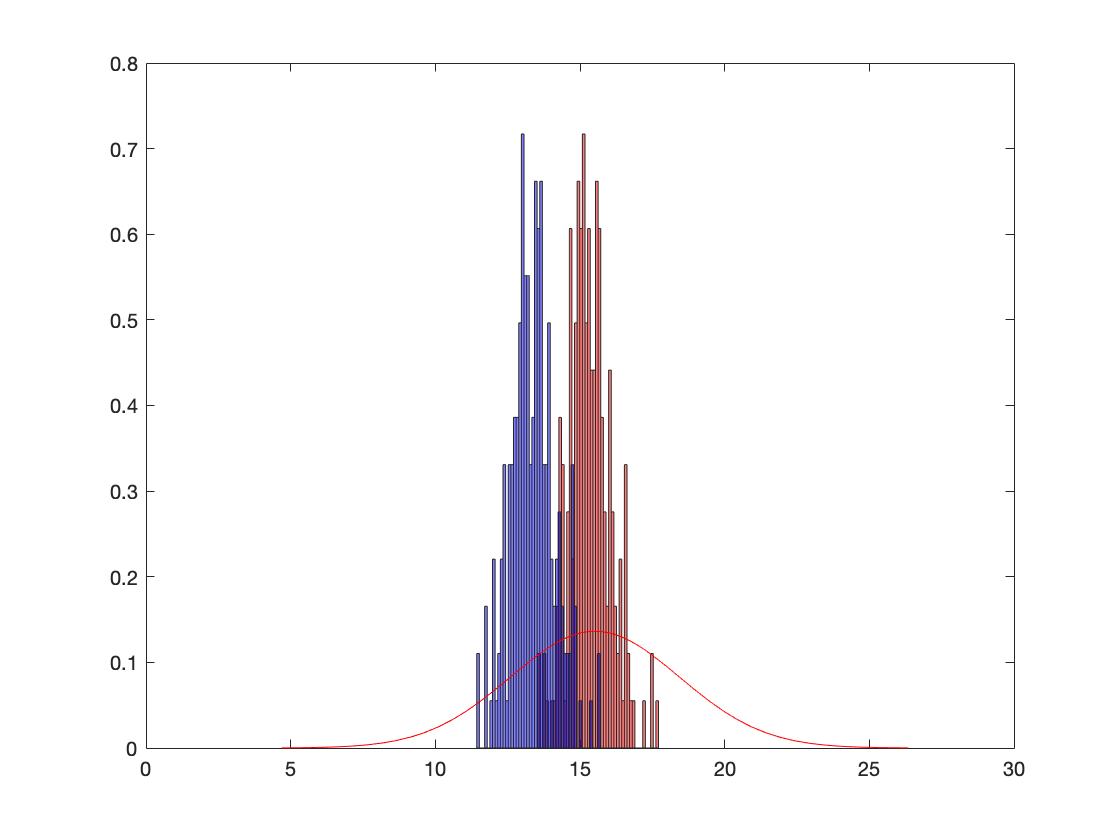}
        \includegraphics[width=2.5in,height=1in]{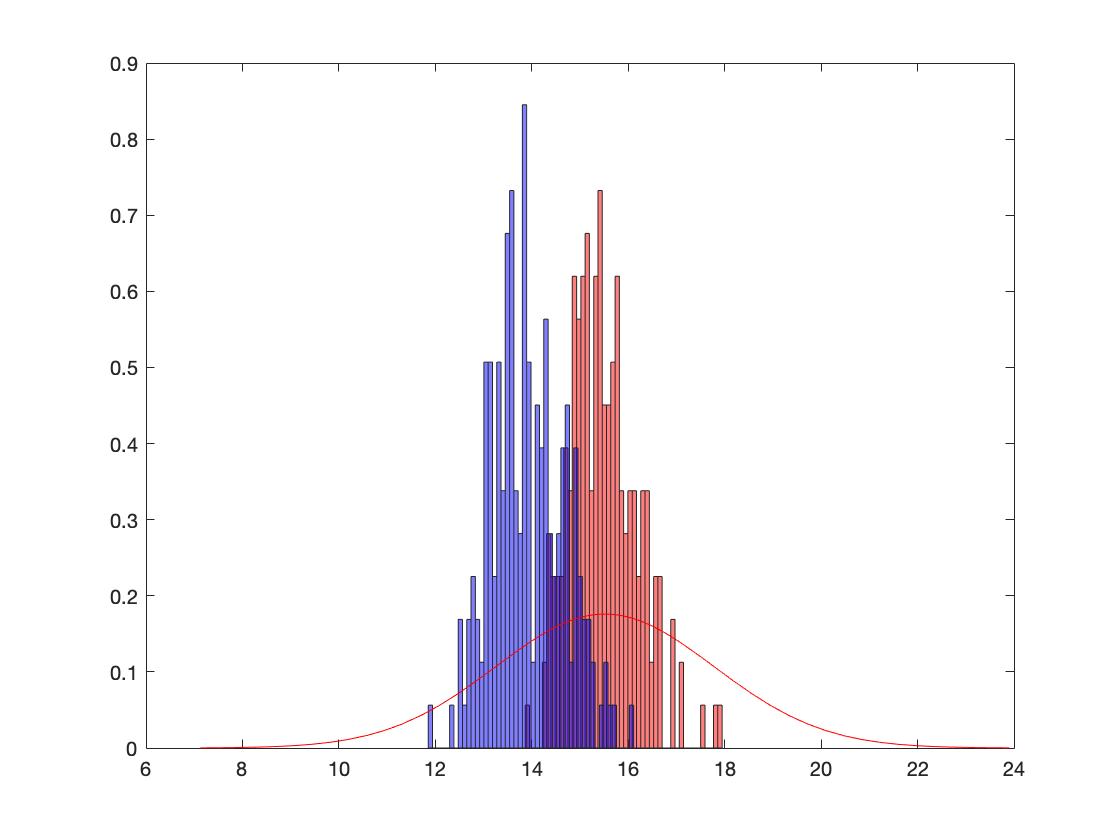}
        \includegraphics[width=2.5in,height=1in]{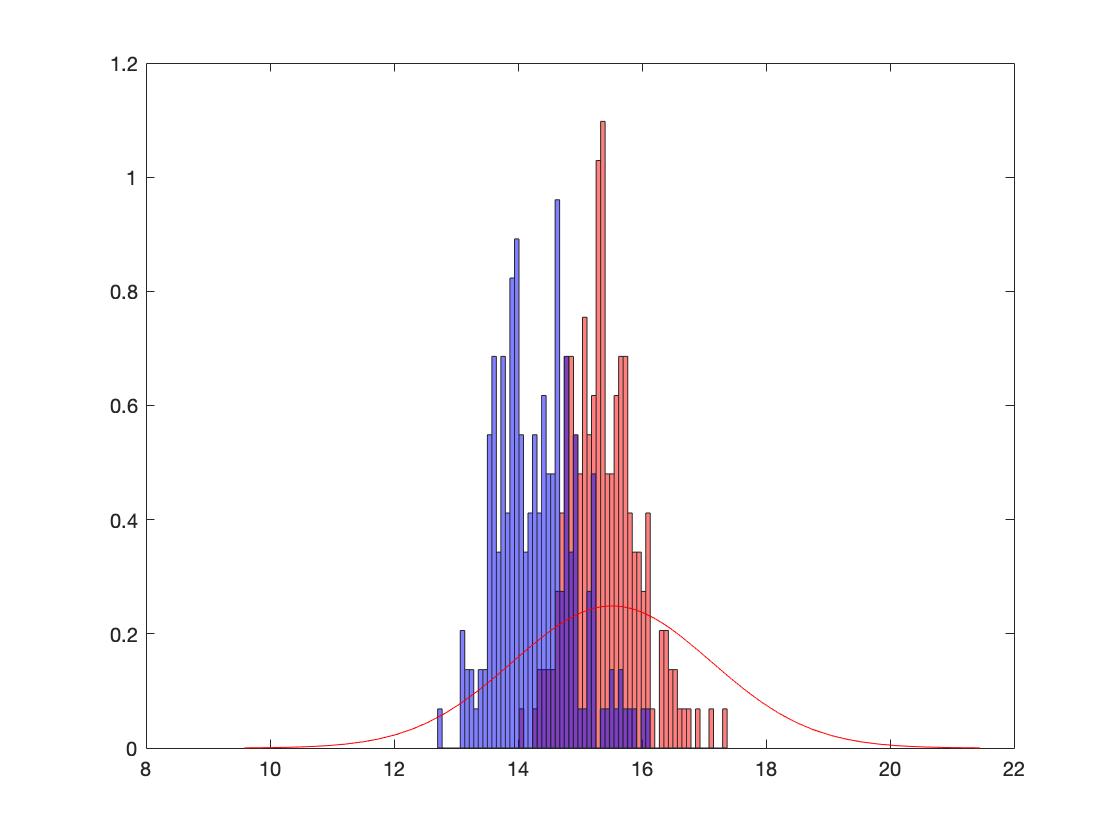}
        \includegraphics[width=2.5in,height=1in]{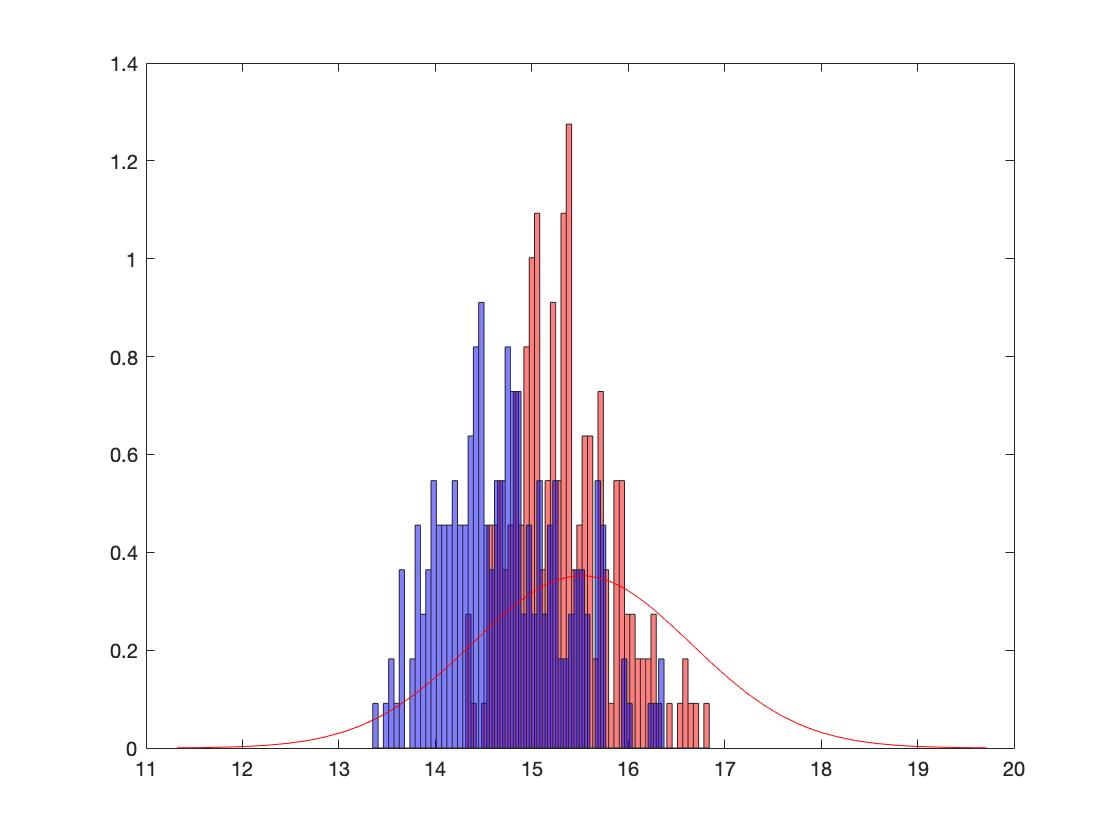}
    \end{center}
    \caption{Density histogram of the distribution of the estimator $\varrho^{(n)}_K$ and $\varrho^{(n)}$ with sample sizes of 30, 50, 100, and 200 (arranged clockwise).}
\end{figure}

 We use the bandwidth calculated according to the formula $1.06\hat{\sigma}n^\frac{1}{5}$, where $\hat{\sigma}$ is the estimated standard deviation of the data. As the sample size increases, both the empirical estimator (blue) and the uniform kernel estimator (red) progressively converge towards the true standard Brownian process depicted by the solid red curve. It is worth noting that the uniform kernel estimator exhibits less bias compared to the empirical estimator. This is our motivation to consider this type of estimate as frequently in practice, obtaining data is costly and we need to work with small samples.
 
\subsection{Comparison of two risk measures}\label{5.2}

Within the following two subsections, we present two primary applications of the Multivariate Central Limit Theorem (CLT) for smoothed estimators. Initially, we delve into the comparative analysis of two portfolios, denoted as $X^1$ and $X^2$, even if their independence is not assured. This scrutiny aims to discern the relative levels of risk attributed to each portfolio while employing a consistent measure of risk derived from sample-based approximations. Subsequently, we turn our attention to a unified portfolio and try to elucidate disparities arising from the utilization of distinct risk measures.\\

First, we consider the difference in risk for two variables. 
Let $X^{1}$ and $X^{2}$ be random variables in $\mathbb{R}^{m}$, not necessarily independent. Let $U \in \mathbb{R}^{2m}$, $X = (X^{1},X^{2})$ and let $\tilde{X}^{1}$ and $\tilde{X}^{2}$ be available random samples of size $n_1$ and $n_2$, respectively. We consider a composite risk functional $\varrho[u,\cdot]$ defined in \eqref{d:rho-vector} and such that the assumptions of Theorem ~\ref{multivariateclt} are satisfied.
Recall
\begin{equation}
 \varrho[X] = \mathbb{E} \left[ g_{1}\left( \mathbb{E} \left[  g_{2}\left( \mathbb{E} \left[ \cdots g_{k}\left(\mathbb{E}\left[ g_{k+1}\left( X \right)\right],X\right)\right]\cdots,X \right)\right],X\right)\right]. 
\end{equation}
We introduce additional function $g_0:\Rb^2\times\Rb^{2m}$, defined as follows: $g_0(\eta,X) = \eta_1 -\eta_2$ and observe that Theorem ~\ref{smoothclt} implies for $n=\min(n_1,n_2)$
\begin{equation}
\sqrt{n}\left( \varrho^{(n_1,J)}[\tilde{X}^{1}] - \rho^{(n_2,J)}[\tilde{X}^{2}] - \varrho[X^{1}] + \varrho[X^{2}]\right) \dto N\big(0, (1,-1)C^\top \Sigma_{g}C(1,-1)\big).
\end{equation}
This allows us to conduction statistical comparison for the two measures of risk. 

In our numerical experiments,  we consider independent identically distributed observations $Y_i, i = 1,\cdots,n$ from a normal distribution $Y \sim \mathcal{N}(20, 5)$. In the model \eqref{d:inverse}, we set the parameters as $c = 20$ and $p = 2$ and compare $\rho(X) - \rho(Y)$, where $X$ has the normal distribution $\mathcal{N}(10,3)$. The difference in the risk values is given by $\rho(X) - \rho(Y) = -11.6052$. The histogram is overlaid on the distribution $N(-11.6052, \frac{16.032^2}{n} + \frac{20.6972^2}{n})$, see in Figure 2.

It has been observed that as both $n$ and $m$ increase, the resulting outcome becomes increasingly closer to the theoretical distribution. The simulation study's findings indicate that, even when dealing with limited sample data in a vector-valued scenario, the uniform kernel estimator outperforms the empirical estimator in terms of accuracy. We use the same sample to estimate the risk with the smooth estimator and empirical estimator, aiming to address the challenge of limited sample size and enhance the quality of our estimation.

\begin{figure}
    \begin{center}
        \includegraphics[width=2.5in,height=1in]{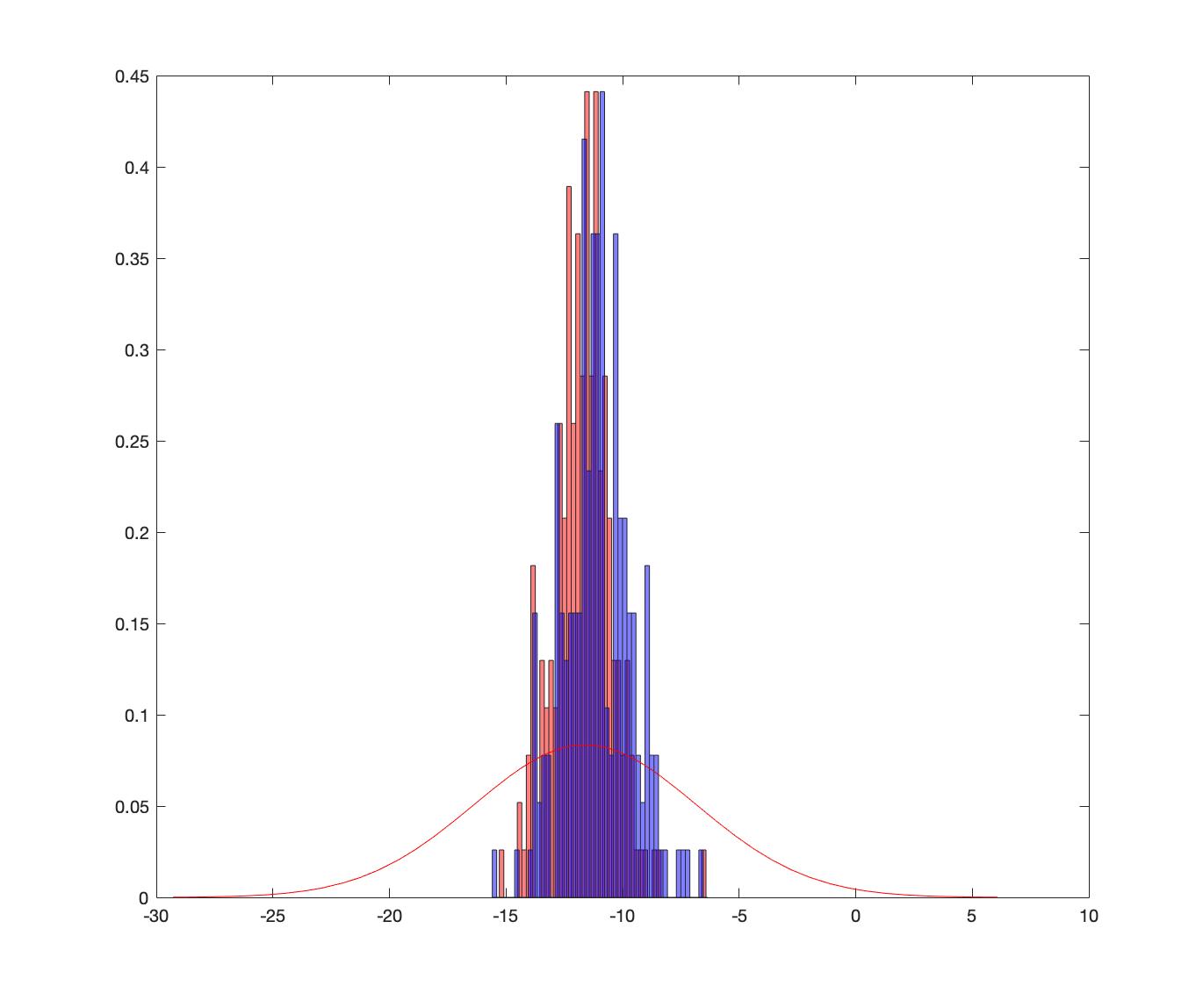}
        \includegraphics[width=2.5in,height=1in]{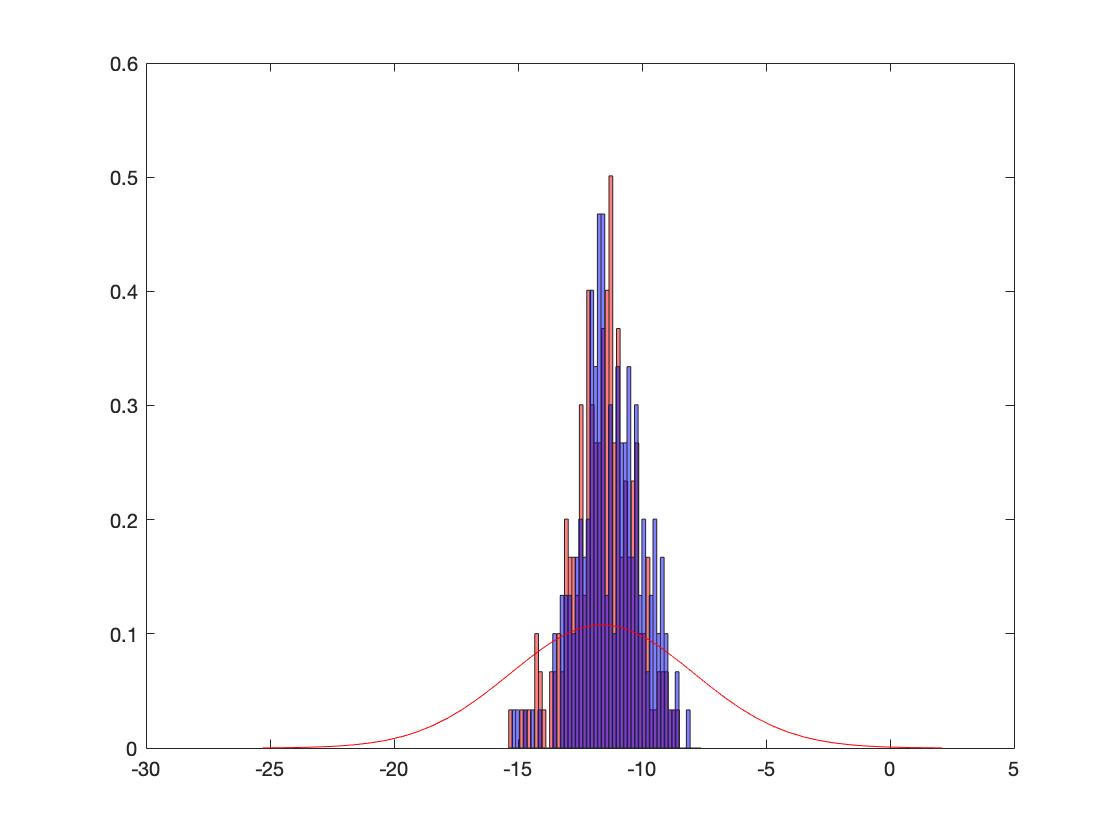}
        \includegraphics[width=2.5in,height=1in]{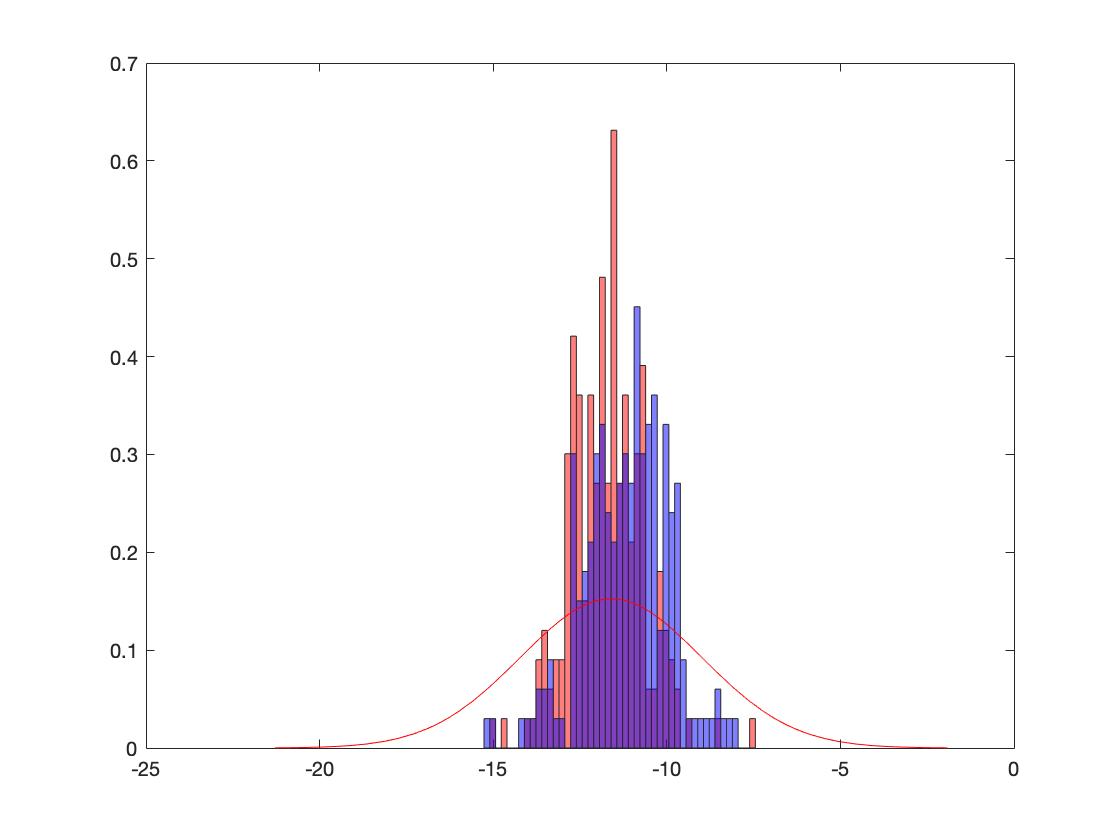}
        \includegraphics[width=2.5in,height=1in]{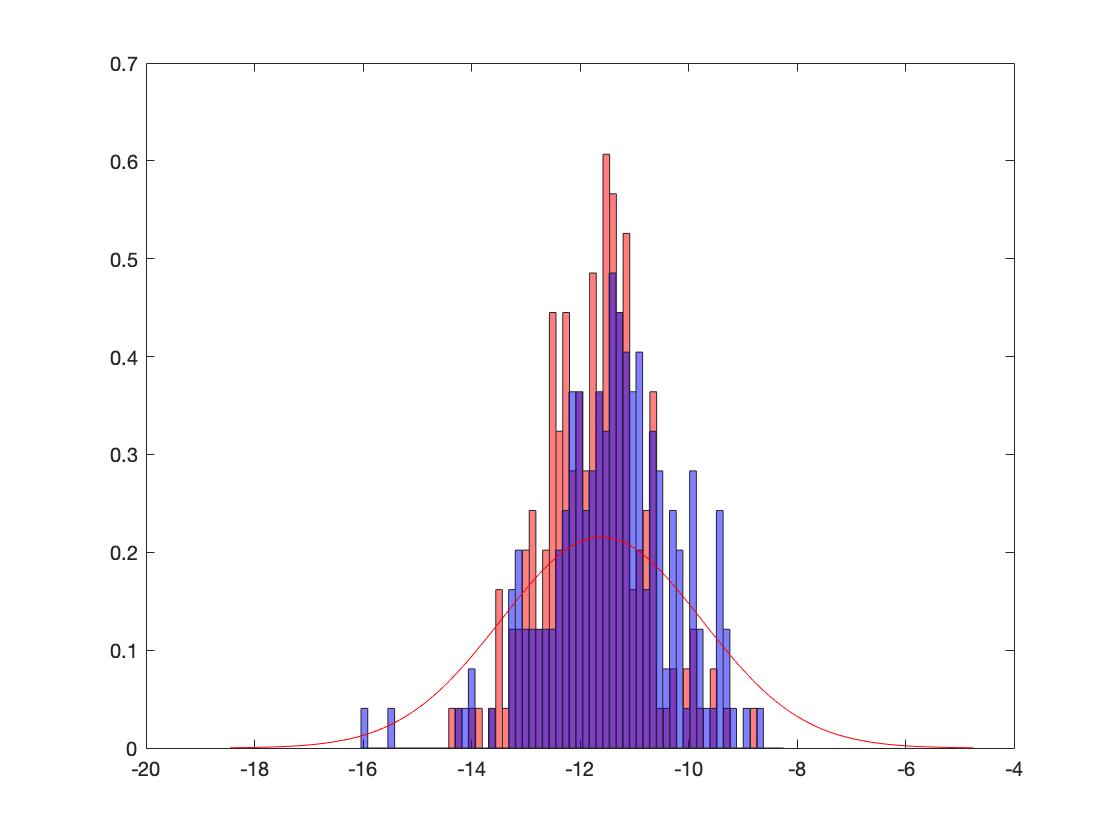}
    \end{center}
    \caption{Density histogram of the distribution of the estimator $\varrho^{(n)}_K$ and $\varrho^{(n)}$ with sample sizes of 30, 50, 100, and 200 (arranged clockwise).}
\end{figure}

We also note that in the same way, we could compare two distinct risk measures $\varrho_{1}[X^{1}] - \varrho_{2}[X^{2}]$ should this be desired.

\subsection{Central Limit Theorem for Systemic Risk}

Our study offers insights into addressing the issue of evaluating risk for a complex distributed system. 
Each individual risk is given by the following composite form:
\begin{equation}
\label{eq_75}
 X_R(i)=\varrho_{i}[X_{i}] = \mathbb{E} \left[ f_{1}^i\left(\mathbb{E} \left[ f_{2}^i\left( \mathbb{E} \left[ \cdots f_{k}^i\left(\mathbb{E}\left[ f_{k+1}^i\left( X_{i} \right)\right],X_{i}\right)\right]\cdots,X_{i} \right)\right],X_{i}\right)\right] 
\end{equation}
If the systemic risk is evaluated by a nonlinear aggregation function $\Lambda:\Rb^m\to \Rb$, we define the total risk as follows:
\begin{equation}
\label{d:systemic-rho-Lambda} 
\begin{gathered}
 \varrho[X] = \Lambda\big(\mathbb{E} \left[ g_{1}\left( \mathbb{E} \left[  g_{2}\left( \mathbb{E} \left[ \cdots g_{k}\left(\mathbb{E}\left[ g_{k+1}\left( X \right)\right],X\right)\right]\cdots,X \right)\right],X\right)\right]\big), \\
 \text{i.e. }\; \varrho[X]= \Lambda\big(\varrho_{1}(X^{1}),\dots,\rho_{\ell}(X^{\ell})\big). 
 \end{gathered}
\end{equation}
Here the functions $g_i$ are defined as in section~\ref{s:MCLT}.
Hence, we obtain the following statement.
\begin{corollary}
Assume that the conditions of Theorem \ref{multivariateclt} are satisfied and the function $\Lambda$  is Hadamard directionally differentiable,
then the systemic risk satisfies the following central limit formula:
\[
\sqrt{n}\;\left[ \Lambda \big(\varrho_{\mu,1}^{(n)},\dots , \varrho_{\mu,\ell} ^{(n)}\big)  - \varrho [X] \right] \dto \Lambda'\left[0;\xi_{1}(W)\right]. 
\]
\end{corollary}
If the aggregation is obtained via an outer coherent risk measure $\varrho_0$ as described in section~\ref{s:MCLT}, then we can observe that we can evaluate $\varrho_{\rm sys}$ directly with additional composition since we have a finite probability space $(\Omega_\ell,\Fc_c, c)$. We can illustrate this point by using the mean semi-deviation of order $p\geq 1$ as the outer risk measure. Let 
$\varrho[X] = \big(\varrho_{1}(X^{1}),\dots,\rho_{\ell}(X^{\ell})\big)^\top$. 
Thus, we have 
\[
\varrho_{\rm sys}[X] = \langle c,\varrho[X]\rangle + \kappa \Big( \sum_{i=1}^\ell c_i \max\big(0, \varrho_{i}(X^{i})- \langle c,\varrho[X]\rangle\big)^p\Big)^\frac{1}{p},
\]
where $\kappa\in[0,1]$ is a fixed parameter. The following statement holds for the systemic measure of risk obtained in this way. 
\begin{theorem}
Assume that the conditions of Theorem \ref{multivariateclt} are satisfied and the outer risk measure $\varrho_0$ in the definition of $\varrho_{\rm sys} (\cdot)$ is coherent,
then the systemic risk satisfies the following central limit formula:
\begin{equation}
\label{clt-rho-sys}
\sqrt{n}\;\left[ \varrho_{\rm sys}^{(\mu,n)}[X]  - \varrho_{\rm sys} [X] \right] \dto \max_{\zeta\in \partial\varrho_0(0) } \langle \zeta, \xi_{1}(W)\rangle. 
\end{equation}
\end{theorem}
\begin{proof}
Denoting $Y_i^{(n)} = \varrho_{\mu,i}^{(n)}$, we observe that the random variables $Y^{(n)}$ with realizations $Y_i^{(n)}$, $i=1,\dots,\ell$ converges point-wise to the random variable $X_R$ with the speed of $\sqrt{n}.$
Since every coherent measure of risk is convex with respect to its argument, it is also Hadamard directionally differentiable. We obtain
that 
\[
\sqrt{n}\;\left[ \varrho_{\rm sys}^{(\mu,n)}[X]  - \varrho_{\rm sys} [X] \right] \dto \varrho'\left[0;\xi_{1}(W)\right]. 
\]
We use the form of the directional derivative in direction $\xi$ of a coherent measure of risk. It is given by the 
$\max_{\zeta\in \partial\varrho_0(0) } \langle \zeta, \xi \rangle.$ 
Hence, we obtain the form of the limiting distribution as in \eqref{clt-rho-sys}.
\end{proof}
In accordance with the numerical experiment setup described in Subsections \ref{5.1} and \ref{5.2}, we proceed by employing the mean semi-deviation of order $p \geq 1$ as the outer risk measure. Specifically, let $\varrho[X] = \big(\varrho_{1}(X^1), \varrho_{2}(X^2)\big)^\top$, where $X^1 \sim \mathcal{N}(10, 3)$ and $X^2 \sim \mathcal{N}(20, 5)$. For the parameters $c_1 = c_2 = 0.5$, $p = 2$, and $\kappa = 0.5$, we compute $\varrho_{\text{sys}}[X]$ to be approximately $23.3704$. By comparing the empirical and uniform kernel estimators, we can derive the following asymptotic behavior: It is observed that the smooth estimator exhibits less bias, with its mean value closer to $23.3704$, compared to the empirical estimator when the outer mean semi-deviation risk measure is applied.

\begin{figure}
    \begin{center}
       \includegraphics[width=2.5in,height=1in]{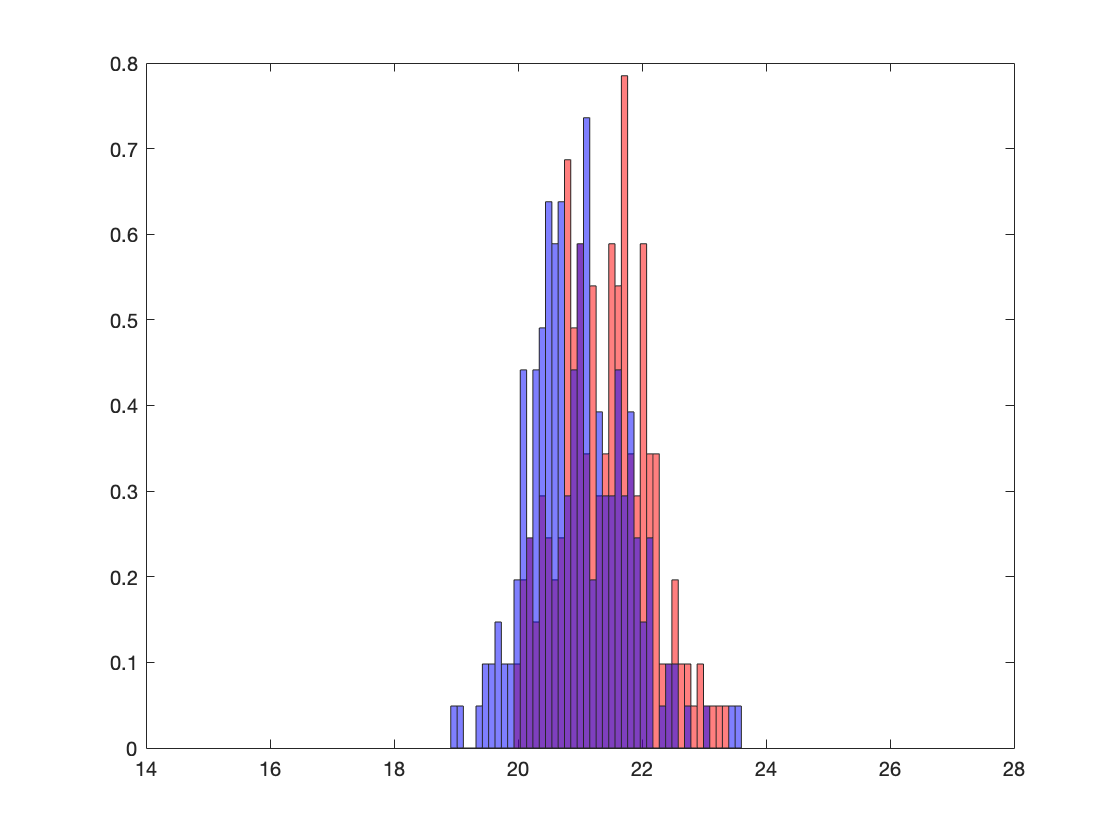}
        \includegraphics[width=2.5in,height=1in]{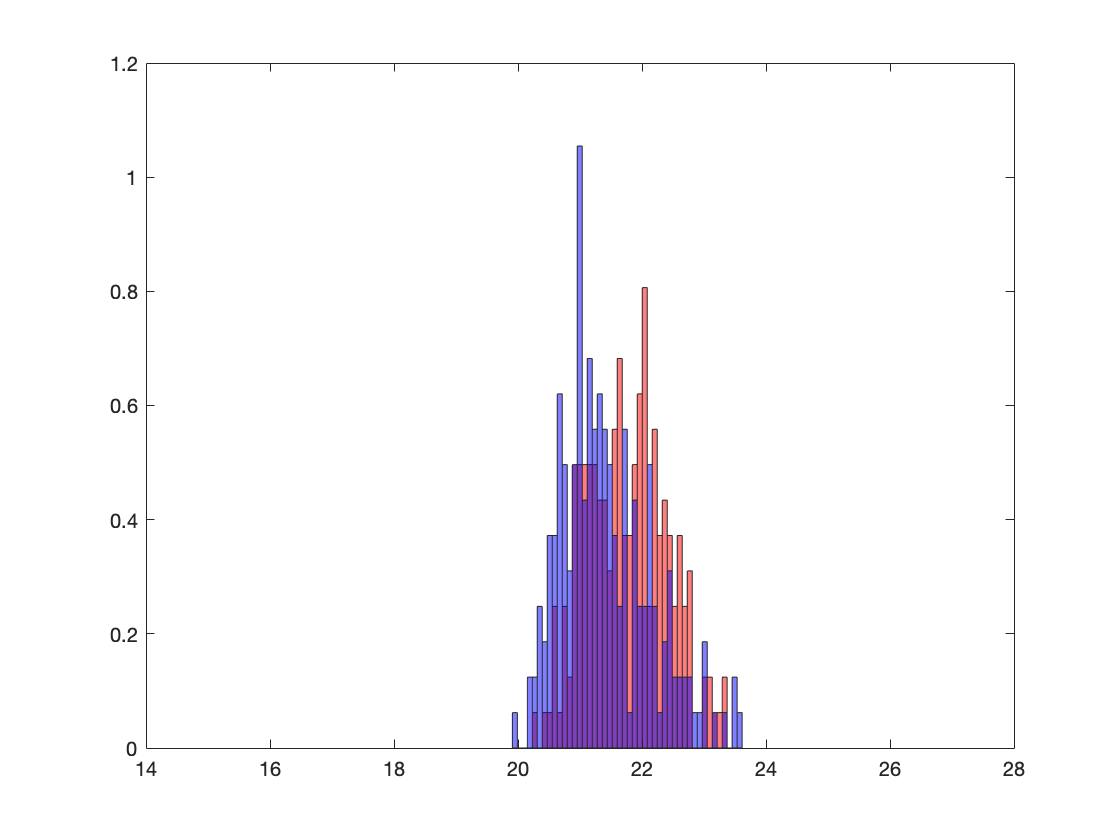}
        \includegraphics[width=2.5in,height=1in]{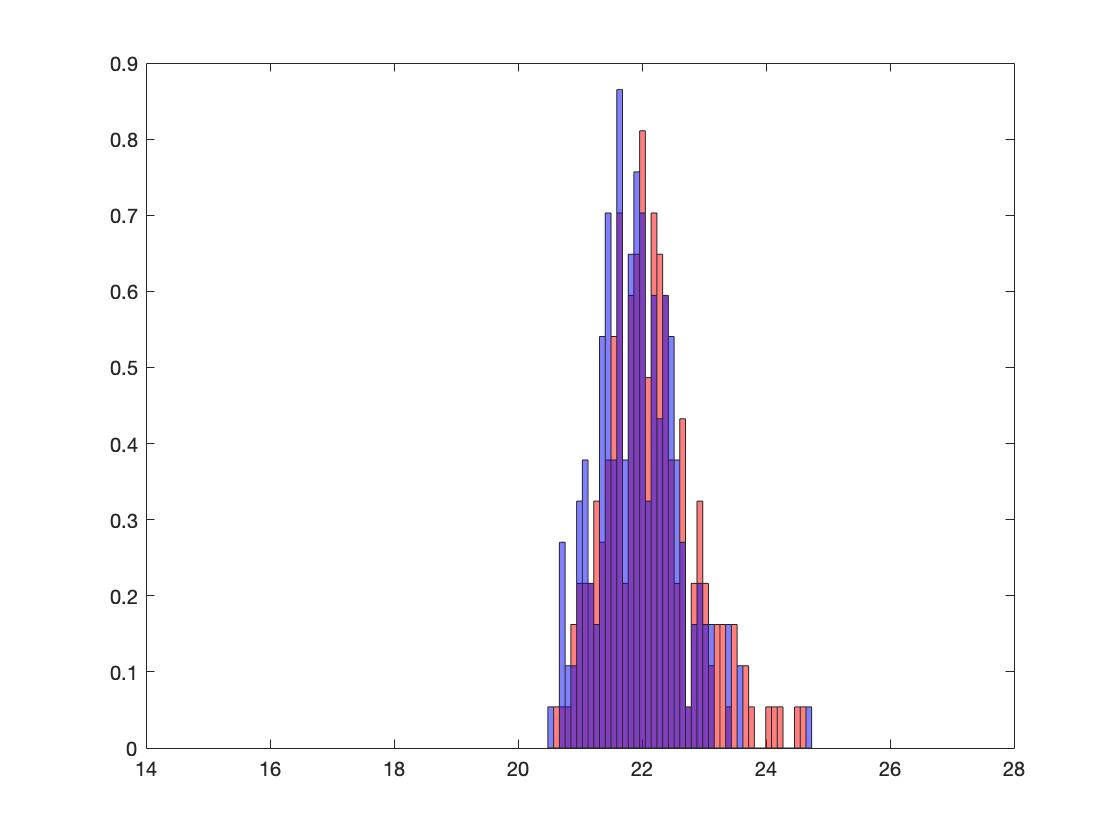}
        \includegraphics[width=2.5in,height=1in]{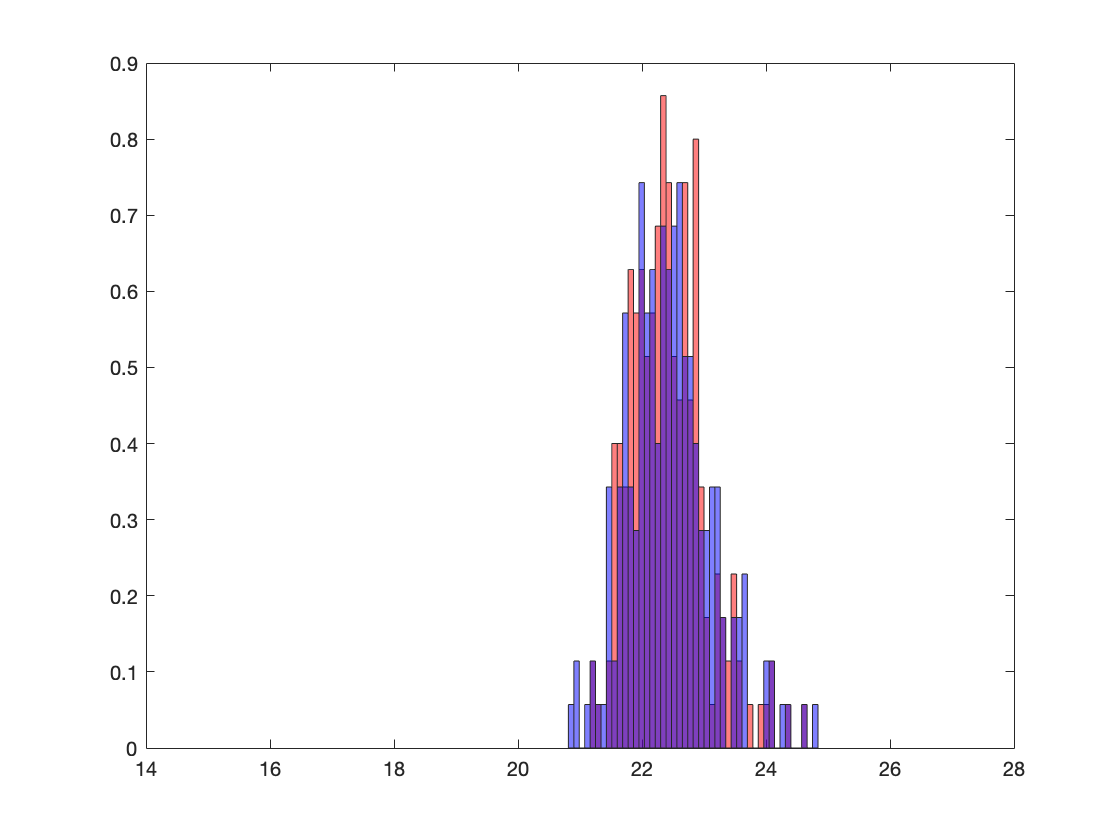}
    \end{center}
    \caption{Density histogram of the distribution of the systemic risk $\varrho^{(K,n)}_{\rm sys}[X]$ and $\varrho^{(n)}_{\rm sys}[X]$ with sample sizes of 30, 50, 100, and 200 (arranged clockwise).}
\end{figure}

\section{Conclusions}
\label{s:conclusions}

In conclusion, our paper makes significant contributions in several key areas. We introduce a novel central limit theorem for composite risk functionals that incorporate mixed estimators, encompassing both smoothing and empirical methods. Additionally, we extend our analysis to  multi-variate measures, providing insights into the verification of assumptions crucial for the central limit formulae. This extension proves valuable for the statistical estimation of systemic risk as well as in the context of statistical tests aimed at comparing levels of riskiness. Furthermore, we specialize our central limit theorem to the application of kernel estimator showing conditions for the bandwidth behavior. Our simulation study demonstrates that the adoption of a smooth estimator yields reduced bias compared to an empirical estimator, a trend observed across both univariate and multivariate scenarios, particularly in addressing challenges associated with small sample sizes.

\bibliographystyle{plain}
\bibliography{multivariate_composite_functionals_arvix}

\end{document}